\documentclass[preprint,12pt,authoryear]{elsarticle}
\usepackage{graphicx}%
\usepackage{subcaption}
\usepackage{multirow}%
\usepackage{wrapfig}
\usepackage{amsmath,amssymb,amsfonts}%
\usepackage{amsthm}%
\usepackage{mathrsfs}
\usepackage[title]{appendix}
\usepackage{xcolor}
\usepackage{textcomp}
\usepackage{manyfoot}
\usepackage{booktabs}
\usepackage{algorithm}
\usepackage{algorithmicx}
\usepackage{algpseudocode}
\usepackage{listings}
\usepackage{nicefrac}

\usepackage{ulem}
\usepackage{xcolor}
\usepackage{multirow, lipsum}
\usepackage{enumerate}
\usepackage{booktabs}
\usepackage{subcaption}
\usepackage{physics}
\usepackage{url}
\usepackage{graphicx, subcaption}
\usepackage{siunitx}
\usepackage{import}
\usepackage{placeins}
\usepackage{amssymb}
\usepackage{empheq}
\usepackage{ulem}
\usepackage{algorithm}
\usepackage{float}
\usepackage{multicol}
\usepackage[english]{babel}
\usepackage{pifont}
\usepackage{footnote}
\usepackage{enumitem}
\usepackage{hyperref} 
\usepackage{cleveref}

\newcommand{\hide}[1]{}
\newcommand{\eop}{\hfill$\square$}
\newtheorem{theorem}{Theorem}
\newtheorem{example}{Example}
\newtheorem*{proof}{Proof}

\newtheorem{corollary}{Corollary}
\newtheorem{definition}{Definition}
\newtheorem{lem}{Lemma}
\newcommand{\us}{u^{s}}

\newcommand{\NE}{\mathcal{N}_{\alpha}}

\newcommand{\NEc}{\mathcal{N}_{1}}

\newcommand{\support}{\mathcal{S}}
\newcommand{\Au}{{\mathcal A}^R_u(\mu)}
\newcommand{\Amu}{{\cal {A}}^H(\mu) }

\newcommand{\MP}{{\cal M}_P}
\newcommand{\MI}{{\cal M}_I}
\usepackage{bm}

\newcommand{\HA}{\mathcal{H}_\alpha}
\newcommand{\iH}{i^{H}}
\newcommand{\A}{\mathcal{A}}
\newcommand{\M}{\mathcal{M}}
\newcommand{\B}{\mathcal{B}}
\renewcommand{\P}{\mathcal{P}}
\newcommand{\z}{{\bf z}}
\newcommand{\x}{{\bf x}}
\newcommand{\ei}{\nu}
\newcommand{\sutil}{u^s}

\journal{Games and Economic Behavior}

\begin{document}
\begin{frontmatter}
\title{Games with Rational and Herding Players}

\author[iitb]{Raghupati Vyas\corref{cor1}}
\ead{raghupati.vyas@iitb.ac.in}

\author[inria]{Khushboo Agarwal}
\ead{khushboo.agarwal@inria.fr}

\author[inria]{Konstantin Avrachenkov}
\ead{k.avrachenkov@inria.fr}

\author[iitb]{Veeraruna Kavitha}
\ead{vkavitha@iitb.ac.in}

\cortext[cor1]{Corresponding author}
\address[iitb]{IEOR, IIT Bombay, Powai, Mumbai 400076, Maharashtra, India}
\address[inria]{Inria Sophia Antipolis, 2004 Route des Lucioles, Valbonne 06902, France}

\begin{abstract}
Classical game theory is a powerful framework to analyze the strategic interactions among rational players.
However, in many real-life scenarios, players choose actions based on their inherent natural tendencies rather than deliberate reasoning.
In this paper, we develop an analytical framework to study large population games with an $\alpha$-fraction of rational and $(1-\alpha)$-fraction of herding (irrational) players. We introduce a new notion of equilibrium called $\alpha$-Rational Nash Equilibrium (in short, $\alpha$-RNE) and discuss its interpretations. 
Some classical equilibria may disappear and some new ones may emerge, but only for smaller $\alpha>0$. 
Interestingly, rational players benefit from the presence of herding and may even achieve utility exceeding the socially optimal level. Even more strikingly, in some cases the herding players also benefit, attaining utility close to the social optimum.   

We further study the effect of herding fraction on system performance through measures like Price of Anarchy (PoA). In transportation networks, a well-known paradox---first studied by Pigou and later by Braess---typically arises from rational decision-making: adding an extra link can reduce overall system efficiency. The analysis of our model points to a different conclusion. When a substantial fraction of users exhibit herding behavior, introducing a new link can increase efficiency, provided herding choices can be suitably influenced. The gains are larger when the herding fraction is higher and/or congestion is lower. By contrast, if herding decisions cannot be influenced, the additional link may become detrimental. We also study a bandwidth sharing game where the herding tendencies improve the system efficiency. 

Finally, we discuss mechanism or influence design in the presence of herding. The expanded set of equilibria offers new opportunities,
however it also increases the risk of undesirable outcomes when influence cannot be created.
\end{abstract}


\end{frontmatter}

\section{Introduction}
Classical game theory has been a powerful framework to analyze the strategic interactions among rational and intelligent players (see works by the pioneers, such as \cite{neumann_morgenstern_1944}, \cite{nash1950equilibrium}). The rational players choose the strategies that maximize their own payoff, given (or anticipating) the strategies of other players. A player who can accurately compute and choose the best strategies to maximize its payoff is said to be an intelligent or rational player.

As time progressed, the classical theory started facing criticism for its strict assumptions about the `intelligence' or `rationality' of the players---in real-life scenarios, players often deviate from these idealized assumptions, maybe due to the limited reasoning powers or due to the lack of information, etc. As a result, there is a shift towards capturing more accurately the actual human behavior, which has paved the way for fields like behavioral game theory (examines the impact of other factors on human decisions, see \cite{ camerer2004behavioural},   \cite{camerer2011behavioral}), behavioral economics (impact of such decisions on economics, see \cite{thaler2018cashews}), and neuro-economics (combines neuroscience, psychology, and economics to uncover the biological basis of irrational choices of players, see \cite{schultz2008introduction}), etc. 

Behavioral game theory, in particular, seeks to study and bridge the gap between theoretical predictions and observed behaviors by incorporating the cognitive and social limitations of the players. Majority of the work in this field is experimental, which tries to capture the fact that individuals cannot think infinitely, may make errors in calculating their utility, or have utilities influenced by social norms and preferences, etc., see the monograph \cite{ camerer2011behavioral}. On the theoretical front, the majority of the literature either considers players that are not far-sighted (e.g., \cite{sandholm2010population}) or considers bounded rational players (e.g., \cite{gigerenzer2002bounded}) that make mistakes in computing their utilities.  For example, the authors in \cite{mckelvey1995quantal} introduce an equilibrium concept known as the  Quantal Response Equilibrium (QRE) that accounts for errors in players' expected utility. This framework assumes players choose their actions based on 
imperfectly estimated utilities with random errors.
The well-known example of QRE is the logit equilibrium, which arises when these errors are distributed according to a bell-shaped curve (see for example \cite{anderson2002logit}). On the other hand, the work in \cite{eliaz2002fault} with subsequent analysis in  \cite{vasal2020alpha}, \cite{vasal2020fault} examine equilibrium concepts involving faulty (kind of irrational) players. These equilibria additionally incorporate stability against deviations by any subset of faulty players. 


\hide{However, a more common scenario observed in real-world examples is the presence of players making choices without any rationality  ---   some mimic others, some avoid crowds,  some make random decisions, or some are loyal to certain choices (e.g., social norms, or loyalty to a brand), etc. There is relatively limited literature that includes the `irrational' players, and we aim to bridge this gap with respect to herding players who blindly mimic the majority.}

However, in many real-world scenarios,   the players choose actions based on their inherent natural tendencies without any iota of rationality (basically, they do not optimize a certain utility function,   with or without estimation errors, to decide their choices).  One can observe such behaviors in a large fraction of decision-makers in a variety of applications, all the more in large population games. Herding is one of the most commonly observed behavioral trends where the individuals follow the majority. 
A few instances where such a trend can be seen are:  among a group crossing the road near a traffic light, the majority just follow the crowd;  in stock markets where the individuals often choose popular options; in e-commerce platforms where similar choices are observed.  
Other major behavioral trends are avoiding the crowd (e.g., choosing a less crowded restaurant despite it being less popular, or selecting a less busy time to visit public places), loyalty (e.g., repeatedly purchasing a specific brand), or even considering random choices, etc. In summary, such choices are more of a tendency, rather than the choices made after meticulous reasoning.

To the best of our knowledge, existing work does not study the games involving a mixture of rational players and those with `irrational-natural-tendency', as discussed above. We strive to overcome this limitation by considering herding behavior among decision-makers. To be precise, the `herding players' in our framework simply choose the action that has been chosen by the majority, without giving any importance to the utilities.

There are strands of literature that discuss herding behavior in various other contexts. 
In
\cite{banerjee1992simple}, \cite{morone2008simple}, the authors present a sequential decision-making model where players, with partial information, also consider the actions of the previous players into their decision-making. They may ignore their own signals based on the observations. The authors show that the players exhibit herding behavior at the resultant equilibrium; at such an equilibrium, the value of the available information can be nullified. 
\cite{eyster2009rational} consider a similar study and illustrate similar results---partial information and beliefs can lead to herding of non-optimal policies (they refer to it as naive herding). 
Thus, in these set of papers, 
bounded-rational herding can lead to non-optimal choices, while, in our work,  reduced  rational choices (coupled with follow the majority trend by the rest)   can lead  to herding of optimal choices; we, in fact, observe that the price of anarchy reduces in many examples/scenarios. 

The second and more important contrast is that our study investigates the outcomes arising from the interplay between strategic decisions  (of probably a smaller fraction of rational players) and herding tendencies.
Such a mixture of population better represents the real-world scenarios and
the study of games with such a mixture   poses several new questions:
\begin{itemize} 
    \item To begin with, what is an appropriate solution concept for such games?
    \item What are the interpretations of such a solution? 
    \item How does the existence (and the proportion) of the herding crowd, influence the utility of the rational players?  How about the overall system performance?
    \item Is it beneficial to be irrational; any such instances of games?
    \item How does the presence (and knowledge) of herding alter the game design? Can the presence of herding tendencies offer designers new opportunities?
\end{itemize}
In this paper, we answer most of the above questions by: (a) developing an appropriate analytical framework that models the interactions between rational and irrational (herding) players; and (b) introducing a new notion of equilibrium, which we named ``$\alpha$-Rational Nash Equilibrium ($\alpha$-RNE)". Here, rational players constitute $\alpha > 0$ fraction of the population.

To begin with, we observe that the $\alpha$-RNE can have similar interpretations (prescription, prediction, limit of some learning dynamics, etc.) as 
 the classical NE (solution with only rational players, i.e., at $\alpha=1$).

When $\alpha$ exceeds a certain threshold, the set of equilibria remains the same as in the classical case (this threshold is proportional to the number of actions). However, for smaller $\alpha$ (below the threshold), some classical equilibria may disappear and some new strategy measures may become equilibria---thus the equilibria in the presence of herding could be significantly different from the case when the presence of herding is neglected (classical case).
Our findings also reveal that the rational
players always benefit in the presence of herding---sometimes derive even better than the social optimal~utility. Importantly, in some cases, the herding players may also derive better, even almost the social optimal~utility. \textit{Perhaps, it is rational to be irrational sometimes!}

We analyze a scenario in the transportation networks that exhibits the well-known Braess paradox (see \cite{braess2005paradox})---the quest here is to consider a more realistic population mixture (with a fraction displaying herding tendencies) and examine again the possibility of the paradox. The answer is still yes---the set of equilibria with and without an additional link (in some scenarios) illustrate that the additional link degrades the performance of all. However, there are subtle differences with and without considering the presence of herding.  
In the example scenario of our paper, we observe that the classical theory underestimates the possibility of the paradox---the equilibria accounting for the presence of herding indicate paradox under a larger set of parameters. 

\hide{s to understand if the paradox is intact even  with a mix of rational and herding travelers. Our findings reveal that rational players may benefit from the presence of herding players. In an example, we showed that under $\alpha$-RNE, the rational players can receive higher utility than the social optimal utility obtained under the classical case ($\alpha=1$); although most of the time, players receive less utility than the social optimal utility under classical NE. Surprisingly, we observed a scenario where all players, including herding individuals, obtain higher utility at $\alpha$-RNE than if all players were rational. This implies that, \textit{``it may be rational to be irrational sometimes"}.} 

We also discuss the design aspects of the games influenced by herding tendencies. In many example  games, it is noticed that 
 the presence of the herding players expands the set of equilibria. This may provide opportunities, as desirable outcomes may get included in the new set of the equilibria---if the designer can strategically create an appropriate and sufficient initial influence, it has a chance to propel the system towards a desirable equilibrium. We discuss a Stackelberg framework to create this influence.  On the other hand, the larger set of equilibria increases the design challenges (may have to be prepared for the worst possible), when it is not possible to create an appropriate initial influence.

In our initial work \cite{agarwal2024balancingrationalitysocialinfluence} presented in a conference, we analyzed a simplified framework with two action choices. In this paper, we extend the analysis to include any finite number of actions.  This important extension made it possible to analyze some important realistic applications:  (i)  the possibility of the well-known Braess paradox in transportation networks is examined,  now in the presence of herding tendencies; (ii) the bandwidth sharing problem with a large number of transmit power choices is analyzed, where the rational, as well as, the irrational players are illustrated to achieve near social optimal utility, as the fraction of rational players reduces. 
 We further include a detailed discussion of various interpretations of the proposed equilibrium concept. We also discuss and illustrate the mechanism (influence) design aspects for the scenarios with herding players.

We next provide 
a brief background on mean-field games and the interpretations of the well-known solution concept, Nash equilibrium. Then, in Section~\ref{section_newnotion}, we propose the new notion of solution, $\alpha$-RNE. Sections \ref{sec_Analysis}-\ref{sec_exampl
es} provide theoretical insights into the new sets of equilibria, study some applications, while Section~\ref{sec_mechanism_design} discusses the mechanism design aspects. Section~\ref{sec_Conclusion} concludes the paper.

 \subsection*{Mean field games, solution and its interpretations}\label{sec_background}

 In games with large number of players, identifying the NE becomes a highly complex task -- one has to optimize the utility function of each player while dealing with a large number of variables. 
In this context, the mean-field game (MFG) theory emerges as a powerful tool; this theory studies the `limit game' as the number of players increases to infinity, under certain `symmetry' and `non-atomic' assumptions: (i) the utility function $u$ and the action set $\A$ are the same for all the players; and (ii) the impact of a single player (or a finite collection of players) becomes negligible in the limit (e.g., \cite{carmona2018probabilistic}).  Basically, the focus shifts from the profile of individual actions   ${\bf a} := (a_i)_{\{i \in \A\}}$ to \textit{the empirical measure or the population measure} $\mu = ( \mu_i )_{\{i \in \A\}}$ of the actions chosen by the players,   where $\mu_i$ represents the fraction of players choosing action~$i$, for each $i$. As a result, the utility function $u$ of any player depends upon the action $i$ chosen by it and the empirical measure  $\mu$ corresponding to the rest of the population. Technically, the utility of any player is a function
 $u: \A \times {\cal P}(\A) \to \mathbb{R}$, where ${\cal P}(\A)$  is the set of probability measures on $\A$; recall here $u$ is the same for all the players, and that $\mu$ is not altered by the action(s) of a single (tagged) or even finitely many player(s). In all, any MFG  is specified by the tuple ${\cal G} := \left < {\cal A}, u \right >$. In this paper, we assume that the action set $\A = \{1, 2, \dots, n\}$ for some $n < \infty$.

The NE for MFGs is defined as the empirical measure $\mu^*$ of the population, that satisfies the following fixed point condition:
 \begin{align}\label{eqn_NE}
    \support(\mu^*) \subseteq \A^R_u(\mu^*), \mbox{ where }&\\
    \support(\mu) = \mbox{support} (\mu) :=  \{i \in \A: \mu_i > 0\} &\mbox{ and } \Au:= {\rm Arg} \max_{i \in {\mathcal A}} u(i, \mu). \nonumber
\end{align}
It has been proved that NEs of finite player games converge to the NE of the mean-field game (MFG-NE) under certain assumptions including asymptotic-symmetry  (see for example \cite{carmona2018probabilistic}).

\subsubsection*{Interpretations of NE}
In the literature, there are several interpretations of the NE (see, e.g., \cite{holt2004nash}, \cite{narahari2014game}), and we discuss a relevant few here: 

 {\bf (i) Prescription:}  If  a moderator/controller suggests actions to a set of players (choosing actions independently)
as per the NE of the underlying game, then none of the players would find it beneficial to oppose the prescription.

 { \bf (ii) Prediction}: In some games,  the players can predict the outcome of the game,  then, the outcome would essentially be a NE. For instance, if iterated elimination of strongly/weakly dominated strategies leads to a unique action profile, then this unique profile is the NE and the players can easily predict such an outcome.

 { \bf (iii) Learning}: Often players interact with each other repeatedly. In such scenarios, each player continuously attempts to adjust its actions based on the history of the actions chosen by the others, in a bid to maximize its own utility (e.g., fictitious play, best response dynamics, replicator dynamics, etc.). In many such situations,  the limit of the learning process is an NE.  However, it is equally important to observe here that some NE can not be the limits, and sometimes one can have other limits (e.g., in replicator dynamics, see \cite{sandholm2010population}). Nonetheless, there exists a variety of dynamics studied previously in the literature that have been proven to converge to an NE under some conditions. Thus, the set of NE is a potential set of limits of the learning processes in many scenarios.

 {\bf (iv) Turn-by-turn dynamics:} Sometimes the players have to choose their actions one after the other, and their utilities are realized after all the players have chosen their actions (e.g., surveys with potential rewards, voting, etc.). The players can observe the empirical measure of the choices of the previous players. This kind of dynamics converges to NEs, again, under some conditions (see, for example, the results in \cite{agarwal2025two} with $\alpha = 1$).

Classical game theory has been a powerful framework to analyze strategic interactions among the rational players. However, as already mentioned, it has limitations due to `rationality' assumption.  We now
 proceed to study the games that involve rational players as well as the players exhibiting herding behavior. In the next section, we return to the above list and discuss parallel interpretations of $\alpha$-RNE.

\section{New notion: $\alpha$-RNE} \label{section_newnotion}
To introduce the new notion of equilibrium, we extend the MFG framework by incorporating the following key distinction --- the population is now composed of $\alpha$ fraction of rational players, while the remaining exhibit herding behavior. We begin with some new notations  --- let $\mu^R_i$ represent the fraction among the rational players that choose action $i$, for any  $i \in \A$, and let the vector $\mu^R:= (\mu^R_i)_{\{i \in \A\}}$  represent the corresponding empirical measure.
Next, we describe the models that capture the interactions of each category of players.

\subsubsection*{Decision-making of rational players}
We model the actions of the rational players as in classical MFGs. Hence, if $\mu$ is the empirical measure of the actions chosen by the entire population\footnote{Recall that the game is described in the mean-field framework, therefore, the action chosen by a single player does not affect the outcome of the game (see \cite{carmona2018probabilistic}).

Given this, it is appropriate to view $\mu$ as the empirical measure corresponding to the `entire' population.},  the best response ($\mu^R$) of any typical rational player against $\mu$ would satisfy the following fixed point equation as in 
\eqref{eqn_NE}:
\begin{eqnarray}\label{eqn_NE_rational}
     \support(\mu^R) &\subseteq& \Au, \text{ where }\\
      \Au &=& {\rm Arg} \max_{i \in {\mathcal A}} u(i, \mu),  \  \ \support(\mu^R) = \{ i \in {\cal A} : \mu^R_i   > 0 \}. \nonumber
\end{eqnarray}
Observe when $\alpha = 1$, we will have  $\mu = \mu^R$ and then  any $\mu$ that satisfies \eqref{eqn_NE_rational} becomes the classical MFG-NE, defined in \eqref{eqn_NE}.

\subsubsection*{Decision-making of herding players}
 The herding players blindly follow the `majority', and do not bother about their utility. In other words, they do not optimize like rational players. One simple way to model the herding behavior of the players is by assuming that they choose the action that has been chosen by majority of the rational players, i.e., they choose the action $i$ which satisfies $\mu^R_i \geq \mu^R_j$ for all $j \in \A$ with $ j \neq i$. 
  
However, the above approach has a limitation.
Typically players cannot distinguish rational from herding players. Hence, we consider a more realistic way of capturing the herding behavior --- 
we assume that the herding players choose an action that has been chosen by the  majority of the players, including the other herding players. To be precise, we assume that each herding player chooses the following action, against $\mu$, the  empirical  measure of the entire population:
\begin{align}\label{eqn_NE_irrational}
        \Amu &:= \min \left\{i: i \in {\rm Arg} \max_{j \in \A} \mu_j\right\}.
\end{align}
Once again at the mean-field limit, the choice of a single (or finitely many) player does not alter $\mu$. 
Further, in the above, for simplicity and tractability of the analysis, we assume that the action with the smallest index in the set, ${\rm Arg} \max_{j \in \A} \mu_j$, is preferred in case of a tie (as is usually considered in game theory literature, e.g.,  \cite{narahari2014game}).
Before proceeding further, we   provide some important remarks on  herding choice, \eqref{eqn_NE_irrational}. 

(i)  
The action chosen by the herding players based on \eqref{eqn_NE_irrational} may not be in the best response to $\mu$. Instead, it reflects a behavioral tendency driven by herding behavior. Thus, it is not a rational choice where they are attempting to leverage upon the efforts (or the decisions) of others as in free-riding (see \cite{feldman2005overcoming}, \cite{hardin2003free}); in the latter case, the players are rational as they try to gain from the decisions of the other players, and would herd only if they find it beneficial.

(ii) The second remark is about the possibility of herding players following (only) themselves; this happens when $\mu_i = 1-\alpha$, at some $\alpha$-RNE, for some action $i$  chosen by herding players. Making provision for such a possibility is important for many reasons, and we immediately list a few. 
  In reality, one can have many more variants of irrational behaviors, for example, some players can choose randomly or can have some preferences, etc.;  these players might be insignificant in the bigger picture (or in a large population) but can be significant enough to lead/mis-lead the herding crowd. 
   Alternatively, a controller can create some initial influence towards 
  some `important' action  (beneficial to the system) which the herding crowd follows, but the rational players may not find it beneficial (we discuss elaborately on such details in Section \ref{sec_mechanism_design}). 
  In all such cases, at the limit, it appears that herding players are following themselves.
Yet another possibility is when the rationals  estimate an action to be beneficial in the beginning (as in a variety of game dynamics, e.g., \cite{sandholm2010population}), irrationals  follow such rational players and themselves, and later the rational players can find an alternate action to be beneficial in view of the new empirical measure $\mu$. So, in all the scenarios discussed above, all herding players choose an action, which is not chosen by others. One such instance of irrationals following themselves is observed in Example \ref{example_Product} provided later.

\subsubsection*{The resultant equilibrium }
Taking into consideration the decisions of all the players, the proportion of players among the entire population,  choosing action $i$ from the action set $\A$ is given by:
\begin{align}\label{eqn_NE_total}
    \mu_i = \alpha \mu^R_i + (1-\alpha) 1_{\left\{i = \Amu\right\}}, \mbox{ for each } i \in \A.
\end{align}
Basically, the fraction of players choosing action $i$ is resultant of the fraction among the rational players choosing action $i$ ($\alpha \mu^R_i$) and the herding players choosing action $i$, if it is the action of the majority, i.e., if $i = \Amu$ (see \eqref{eqn_NE_irrational}); \textit{we henceforth refer action $\Amu$ as herding choice.}
Having defined the actions of both rational and herding players, next, we introduce the new\footnote{
Our framework should not be confused with the Bayesian framework (see \cite{narahari2014game}). In our framework, rational players solve the following (with $\mu^H = \delta(\Amu)$, where $\delta$ is the Dirac measure):
$$
\mbox{Arg max}_{i \in \A} u(i, \alpha \mu^R + (1-\alpha) \mu^H),
$$
where the player responds to the aggregate distribution of actions, treating the population as a single mixed entity. In contrast, under the Bayesian framework, a rational player would solve:
$$
\mbox{Arg max}_{i \in \A} (\alpha u(i, \mu^R) + (1-\alpha) u(i, \mu^H)).
$$
Essentially, our approach assumes the player observes and reacts to the combined influence of the overall population, whereas the Bayesian formulation assumes the player reasons probabilistically over separate population types and integrates utilities accordingly.} notion:
\medskip
\begin{definition}\label{defn_alphaRNE}
For any $\alpha \in (0,1]$, a pair of empirical measures $(\mu, \mu^R)$ is called an \textbf{$\alpha$-Rational Nash Equilibrium}, or in short, \textbf{$\alpha$-RNE}, if it satisfies the following:
\medskip
\begin{enumerate}\rm
 \item[(i)]  $\support(\mu^R) \subseteq \Au$,  see \eqref{eqn_NE_rational}, and
\item[(ii)] $\mu_i = \alpha \mu^R_i + (1-\alpha) 1_{\left\{i = \Amu\right\}} \mbox{ for each } i \in \A$, see \eqref{eqn_NE_total}.
\end{enumerate}
\end{definition}
\medskip
Basically, the rational players act optimally in response to the overall population distribution $\mu$, as captured by part (i). Further, the overall fraction of players $\mu_i$ choosing action $i$ in part (ii) is composed of the fraction of rational players choosing action $i$ (i.e., $\alpha\mu_i^R$) plus all the herding players if action $i$ has been chosen by the majority of the players (i.e., $(1-\alpha) 1_{\left\{i = \Amu\right\}}$). 

 Observe that, similar to NE, the new notion $\alpha$-RNE (Definition \ref{defn_alphaRNE}) is also \textit{stable against unilateral deviations of rational players} while herding players simply follow the majority. This notion can be used as a tool to analyze  more realistic scenarios by taking into consideration the influence of herding tendencies by some decision makers. 
 Let us immediately begin with an example of one such game.

\begin{example}[{\bf Product Selection Game}]
 \label{example_Product}
Consider a game where each player (customer) has to select one among the three available products, i.e.,  $\A = \{1,2,3\}$. Assume that the utility function of the customers, $u(\cdot, \cdot)$, satisfy the following relation:
$$
u(1,\mu) > u(2,\mu) > u(3,\mu), \mbox{ for each } \mu.
$$
Clearly, ${\rm Arg} \max_{a \in \A} u(a, \mu) =\{1\}$ for any $\mu$. 
Hence, from \eqref{eqn_NE}, there exists a unique classical NE,  $\mu_{\mbox{\tiny{classical}}} = (1,0,0)$, for the game.

In majority of such selections,  not all choices are rationally made, rather there is a huge tendency for herding (or inclination towards the choice of the majority).
Now, consider a more realistic scenario where irrationals outnumber the rational players, i.e., say $\alpha \leq \nicefrac{1}{2}$. From \eqref{eqn_NE_rational}-\eqref{eqn_NE_total}, $(\mu_\alpha,\mu^R_\alpha)$ with $\mu_\alpha = \mu_{\mbox{\tiny{classical}}}$ as above and $\mu^R_\alpha = (1,0,0)$ is an $\alpha$-RNE; thus classical NE is also an $\alpha$-RNE. Moreover, there are two additional $\alpha$-RNEs, $\mu_\alpha = (\alpha,1-\alpha,0)$ and $\mu^R_\alpha = (1,0,0)$ or $\mu_\alpha = (\alpha,0,1-\alpha)$ and $\mu_\alpha^R = (1,0,0)$; it is noteworthy that, in both the cases, the irrational customers follow (only) themselves in the equilibrium and the two differ only in the
herding choice. Furthermore,  the extra equilibria exist even for larger fractions of  rationals, in fact  even up to $\alpha \le \nicefrac{2}{3}$, as can be verified using \eqref{eqn_NE_rational}-\eqref{eqn_NE_total}. 

\end{example}
This example illustrates that the set of $\alpha$-RNEs need not be the same as that of the classical NEs. There is a possibility of the emergence of new equilibria in the presence of herding players. This motivates us to investigate more about the aspects related to $\alpha$-RNE, which we will consider in the upcoming sections. In the immediate next, we focus on the interpretations of the new notion.

\subsection{Interpretations of $\alpha$-RNE} 
\label{subsection_interpretation}
In the previous section, we discussed various interpretations of the classical NE.  We now examine whether similar interpretations 
are applicable to $\alpha$-RNE, and hence investigate whether $\alpha$-RNEs can better replace classical NEs for the games influenced predominantly by  herding tendencies.  
Before proceeding, we first define a key set of actions that we will use throughout the paper, the set of `equilibrium-herding choices':
\begin{equation} \label{potential_action_herding}
    \HA := \{i \in \A : \exists  \hspace{0.2mm} \mbox{ an } \alpha \mbox{-RNE } (\mu,\mu^R) \text{ such that } i = \Amu\}.
\end{equation}
This set characterizes all the actions that are relevant to the herding players at equilibrium. In other words, any action outside $\HA$  is not a herding choice at any $\alpha$-RNE and thus cannot be an action chosen by herding players at any equilibrium of the game. 
Furthermore, all rational players are assumed to have the knowledge of $\HA$. We now discuss   the interpretations of $\alpha$-RNEs.

 \textbf{(i) Prescription:}
In certain games, the $\alpha$-RNE can be seen as an appropriate prescription for the players. For instance, consider a scenario where  an external entity, such as a controller or moderator controls the game. The controller can compute all the possible $\alpha$-RNEs. Suppose it wants to prescribe one of the desirable $\alpha$-RNEs, say $(\mu_*,\mu^R_*)$. 
It can successfully prescribe the $\alpha$-RNE, $(\mu_*, \mu^R_*)$  by targeted advertising and because of unilateral stability: (i) it can influence the herding players to align with the action $i^* := \A^H(\mu^*)$, using an appropriate mechanism such as targeted advertisements, or incentive schemes (we provide detailed discussion of such influence design in Section \ref{sec_mechanism_design}); 
(ii) if it is successful in directing the herding crowd towards $i^*$, then the rational players would also comply, since any unilateral deviation would not lead to a better outcome for the rational players (see Definition \ref{defn_alphaRNE}).

\textbf{(ii) Prediction:} Another meaningful interpretation of $\alpha$-RNE is its role in prediction, similar to NE, as discussed in the previous section. In certain games with a mixed population,  the rational players can  predict the outcome in terms of their actions utilizing the available information (common knowledge about the game).
Consider the following set of empirical measures, one for each equilibrium-herding choice $\iH \in \HA$:
\begin{equation}
  \mathcal{E}(\iH) = \{\mu : \mu_{\iH} \geq 1-\alpha\}.  \label{Eqn_mathcalE}
\end{equation}
If, for all $\iH \in \HA$ and $ \mu \in  \mathcal{E}(\iH)$, the following holds:
\begin{equation}\label{eqn_str_dom_eqn}
    u(i,\mu) < u(j,\mu),  \mbox{ for all } j \in \A\setminus\{i\},
\end{equation}
then the rational players will not find action $i$ beneficial and will eliminate it (basically, the action $i$ is dominated  by $j$, irrespective of all the possible equilibrium-herding choices). Consequently, the action  set for rational players   reduces to $\A\setminus\{i\}$.

After this elimination, now based on the reduced action set $\A\setminus\{i\}$ (and the  set of equilibrium-herding choices $\HA$), the rational players might find another action, say $k$, as dominated and will eliminate it as well. This can continue. If \textit{such an iterative elimination process} eventually leads to a unique choice $i^*$, then $i^*$ emerges as the only rational choice at any $\alpha$-RNE. In such scenarios,  one can predict the outcome in terms of rational choices. For instance, in Subsection \ref{subsubsec_Braess_additional_link}, while studying the Braess Paradox example, 
we will note that 
the new link $AB$ emerges as the predicted outcome 
in terms of rational choices, as  
$AB$  is the strictly dominant strategy (irrespective of equilibrium-herding choices) for the rational players. We indeed observe in Corollary \ref{corollary_braess_2}(ii) for  $\alpha \le \nicefrac{2}{3}$  that $AB$ is the only rational choice in both the  $\alpha$-RNEs.

\textbf{(iii) Learning with herding:} 
Similar to the classical scenario, in games with both rational and herding players, the concept of $\alpha$-RNE can also be naturally interpreted from a dynamic perspective. Where rational players continuously adjust their actions based on the observed history of actions taken by both rational and herding players, to maximize their utility; while herding players dynamically adopt the action chosen by the majority of the players.

Dynamic processes such as fictitious play, best response dynamics, and replicator dynamics, can also be extended to accommodate this mixed population framework. We intend to demonstrate that the interplay between the strategic adjustments of rational players and herding behavior converges to 
$\alpha$-RNE, in future.

 \textbf{(iv) Turn by Turn dynamics with herding:} 
In games involving rational and herding players, the new notion $\alpha$-RNE can be seen as the limit of the turn-by-turn dynamics discussed in the previous section. In our recent work \cite{agarwal2025two}, we demonstrated the almost sure convergence of such dynamics to the set of $\alpha$-RNEs under some assumptions and when $|\A| = 2$.

In the next section, we delve deep into this new solution concept ($\alpha$-RNE) in games with multiple actions, where we derive some key insights and simple conditions to identify the $\alpha$-RNEs. Additionally, we analyze utility variations across different scenarios.

Before concluding the section, we would like to make an important comment about 
$\HA$, the set of equilibrium-herding choices. \textit{There is a possibility that some actions are not  herding choices} (i.e., $a \notin \HA$ for some $a$) ---  
by Lemma \ref{Lemma_Halpha} of Appendix \ref{appendix},
such a consequence is possible (if at all) only when the rational population is large with $\alpha > 1-\nicefrac{1}{|\A|}$.  For example, by Corollary \ref{corollary_braess_2}(i) of Subsection~\ref{subsec_Braess_paradox}, while studying  the transportation network with three choices $\A = \{1,2,3\}$, we will observe that $\HA = \{3\} \subsetneq \A  $, if $\alpha > \nicefrac{2}{3}$. On the other hand, with lesser rational population (when  $\alpha \le 1- \nicefrac{1}{|\A|}$), we have 
  $\HA = \A$ implying every action in $\A$ can be a herding choice at some    equilibrium (again by Lemma \ref{Lemma_Halpha}).   

\hide  {\color{red}
Before concluding the section, we would like to make an important comment about 
$\HA$, the set of equilibrium-herding choices. There is a possibility that some actions are not  herding choices (i.e., $a \notin \HA$ for some $a$) --- such a consequence is possible only when the rational population is large with $\alpha > 1-\nicefrac{1}{|\A|}$.  For example, in Corollary \ref{corollary_braess_2}, part (i) of Subsection~\ref{subsec_Braess_paradox}(Braess paradox and herding), if $\alpha > \nicefrac{2}{3}$ then $\HA = \{3\}$ while $\A = \{1,2,3\}$, thus $\HA \neq \A$. On the other hand, by Lemma \ref{Lemma_Halpha} of Appendix \ref{appendix}, with lesser rational population (when  $\alpha \le 1- \nicefrac{1}{|\A|}$), we have 
  $\HA = \A$ which implies that every action in $\A$ can serve as a herding choice at some equilibrium.}

\section{Analysis} 
\label{sec_Analysis}
Consider a population game where $\alpha$-fraction of the players are rational, while the remaining exhibit herding behavior.  From
Definition \ref{defn_alphaRNE},  if $(\mu,\mu^R)$  is an $\alpha$-RNE, then the two are related by $ \mu^R = \mu^R(\mu)  = (\mu_i^R(\mu) )_{i \in \A}$, where the component functions are defined as below:
\begin{equation}
    \label{Eqn_muR_given_mu}
     \mu_i^R(\mu) :=  \frac{\mu_i}{\alpha}, \; \mbox{ } \forall \mbox{ } i \in \A \setminus \{k\}, \; \text{where } k = \Amu, \text{ and } \mu_k^R (\mu) := \frac{\mu_k - (1-\alpha)}{\alpha}.
\end{equation}
In other words, any  $\alpha$-RNE can be defined uniquely in terms of just the overall empirical measure $\mu$,  basically any $\alpha$-RNE can be expressed  as $(\mu,\mu^R (\mu) )$ using only the first component $\mu$.
Thus,
\textit{we subsequently represent $\alpha$-RNEs using only $\mu$ components and let $\NE$  represent the set of $\alpha$-RNEs for any $\alpha \in (0,1]$}:
\begin{align} \label{eqn_set_NE_multiple}
    \NE &= \left\{\mu : (\mu, \mu^R (\mu) ) \mbox{ is an $\alpha$-RNE and $\mu^R (\cdot)$ is as in  \eqref{Eqn_muR_given_mu}}\right\}. 
\end{align}
 Notably, when $\alpha = 1$, $\NE$ is the set of classical MFG-NEs, denoted as $\NEc$.
In the following subsections, we provide a detailed characterization of the set $\NE$ of equilibria.


\subsection{Identification of $\alpha$-RNEs} \label{Identification_MFG_NEs_multiple}
 To proceed, we define certain convenient sets of population measures. 

(i) The set $\MP$ of pure population measures is defined as follows:
\begin{equation} \label{set1_for_NEc}
    \MP := \left\{\mu : \mu_i \in \{0,1\} \text{ and } \sum_{i \in \A} \mu_i = 1 \right\}.
\end{equation}
Population measures from this set support a single action.

(ii) We define the following subset of mixed (impure) population measures where at least two actions yield equal utility for rational players: 
\begin{equation} \label{set2_for_NEc}
     \MI := \left \{\mu :  h_{ij}(\mu) = 0, \mbox{ for all $(i,j)$ such that } i,j \in \support(\mu), i \neq j, \ |\support(\mu)| \geq 2 \right\},
\end{equation}
where the function $h_{ij}(\cdot)$ represents the utility difference between the $i$-th and $j$-th actions, and is given by:  
\begin{equation}\label{eqn_def_h_function}
    h_{ij}(\mu) := u(i,\mu)-u(j,\mu), \mbox{ for all } i, j \in \A \text{ with } i \neq j.
\end{equation}
To begin with, from Definition \ref{defn_alphaRNE} and using  the sets defined in \eqref{set1_for_NEc} and \eqref{set2_for_NEc}, one can immediately observe the following for \textit{the set of  classical MFG-NEs ($\NEc$), i.e., at~$\alpha=1$:}

\begin{enumerate}[label=\textbf{(C.\roman*)}, ref=\textbf{(C.\roman*)}, align=left]
    \item \textit{$\NEc \subseteq \MP \cup \MI$, and conversely,} \label{item:i}
    \item \textit{if $\mu \in \MP \cup \MI$ and $\support(\mu) \subseteq \Au$, then $\mu \in \NEc$.} \label{item:ii}
\end{enumerate}
\medskip
We now proceed towards a detailed \textit{characterization of the set of $\alpha$-RNEs for any fixed $\alpha < 1$}, by introducing two more sets of ($\alpha$-dependent) population measures.

\noindent(i) The set $\M_\alpha$ of population measures, in which  each  action is chosen by strictly less than $1-\alpha$ fraction of the population, is formally defined as:
\begin{eqnarray}\label{eqn_set_M_alpha}
     \M_\alpha := \{ \mu : \mu_i  <  1- \alpha \mbox{ for all } i \in \A\}.
\end{eqnarray}
(ii) The set $\P_\alpha$ containing all the population measures that would have resulted if there was  one   `herding' action (say $i$) and when the rationals find it beneficial to choose   actions other than $i$,  i.e., with  $\mu_i=1-\alpha$ and $\support(\mu^R) = \left (\support(\mu)  \setminus \{i\} \right) \subseteq \Au$, see \eqref{eqn_NE_rational}-\eqref{eqn_NE_irrational}: 
%
\begin{eqnarray}\label{eqn_set_P_alpha}
    \P_\alpha := \left \{ \mu :  \mu_i = 1-\alpha  \mbox{ for } i = \Amu, \text{ and }
        \left (\support(\mu) \setminus \{i\} \right ) \subseteq \Au\right\}.
\end{eqnarray}

Before we proceed, observe that $\M_\alpha$ and $\P_\alpha$ are empty for any  $\alpha > 1-\nicefrac{1}{n}$, recall $n= |\A|$.
Having defined these sets, we now state the main theorem that identifies the set of $\alpha$-RNEs using the set of classical MFG-NEs.
\medskip
\begin{theorem}[\textbf{Identification of $\alpha$-RNEs}]\label{thrm_alphaRNE_multiple}
    Let $\A = \{1, 2, 3, \dots, n\}$. Then, for any $\alpha \in (0,1]$, the following holds, with $M_\alpha$ and $P_\alpha$ defined respectively in \eqref{eqn_set_M_alpha} and \eqref{eqn_set_P_alpha}:
  $$
   \NE = \left (\NEc  \setminus \M_\alpha\right )    \cup \P_\alpha. 
   $$ 
   In particular, if $\alpha > 1-\nicefrac{1}{n}$, we have $ \NE =  \NEc. $
\end{theorem}
\begin{proof} The proof is provided in Appendix \ref{appendix}. \eop
\end{proof}
Thus, in general, neither the set of NEs is contained within the set of $\alpha$-RNEs nor vice versa.  The important implications  for any $\alpha < 1$ are as follows:

(i)   if the fraction of \textit{rational agents is sufficiently high}, i.e., if $\alpha > 1-\nicefrac{1}{n}$, \textit{then the presence of herding players does not influence the set of equilibria, i.e., $\NE = \NEc$} (recall $\M_\alpha=\P_\alpha = \emptyset$ for such $\alpha$);   

(ii)   the situation is drastically different when $\alpha \leq 1- \nicefrac{1}{n}$, i.e., if the fraction of rationals is smaller --- certain classical equilibria might no longer exist (those in $\M_\alpha$), and new $\alpha$-RNEs can emerge (those in $\P_\alpha$); and

(iii) surprisingly, as $n$   (number of actions) increases, even a small fraction of herding players $(1-\alpha) \lessapprox  \nicefrac{1}{n}$ can exert a significant influence on the set  of  equilibria, $\NE$.


We now digress slightly to a special case with two actions, in a bid to 
  better illustrate the results of Theorem \ref{thrm_alphaRNE_multiple}.   The analysis of such a simpler game is already presented in our recent initial conference paper \cite{agarwal2024balancingrationalitysocialinfluence}, we would also demonstrate that the results of  Theorem \ref{thrm_alphaRNE_multiple} match with those in \cite[Theorem 2]{agarwal2024balancingrationalitysocialinfluence}.

\subsubsection{Game with two actions (n=2)} \label{subsub_sec_games_two_action}
With $n=2$, the two relevant empirical measures can be represented using simpler notation: $\mu = (z, 1-z)$ and $\mu^R = (y, 1-y)$; here $z= \mu_1 $, and $y=\mu^R_1$ represent the components corresponding to the first action.
Accordingly, the utility function $u(\cdot, \mu)$ can be expressed as $u(\cdot, z)$, see~\eqref{Eqn_muR_given_mu}. Observe that the set of $\alpha$-RNEs ($\NE$) can be represented only in terms of $z$ based on \eqref{Eqn_muR_given_mu}-\eqref{eqn_set_NE_multiple}: 
\begin{eqnarray*}
    \NE &=& \left\{z : (z, y^*(z)) \text{ is an } \alpha\text{-RNE} \right\}, \text{ where } \\
    y^*(z) &:=& \frac{z}{\alpha} 1_{\{z < \frac{1}{2}\}} 
    + \left(\frac{\alpha-(1 - z)}{\alpha}\right) 1_{\{z \geq \frac{1}{2}\}}.
\end{eqnarray*}
Here, the sets defined in \eqref{set1_for_NEc}-\eqref{set2_for_NEc} simplify to the following:
\small{ 
\begin{equation*}
        \mathcal{M}_P = \{(1,0), (0,1)\}, \mathcal{M}_I = \{(z, 1-z) : h_{12}(z) = u(1, z)-u(2,z) = 0,\, z \in (0,1)\}.
\end{equation*}}
Now, \ref{item:i} remains as before, while, using the above, \ref{item:ii} simplifies to:  

(a) $0 \in \NEc$ only if $h(0) \leq 0$, and $1 \in \NEc$ only if $h(1) \geq 0$; and  

(b) any population measure that equalizes the two utilities $u(1,z)$ and $u(2,z)$  is  an equilibrium at $\alpha =1$,  i.e.,  $\mathcal{M}_I \subseteq \NEc$.
\hide{
Similarly, the sets defined in \eqref{eqn_set_M_alpha}-\eqref{eqn_set_P_alpha} simplify:  
if $\alpha \le \nicefrac{1}{2}$, we have,
\begin{eqnarray*}
 \P_\alpha =
\left\{
\begin{array}{ll}
    \alpha, & \text{if } h(\alpha) > 0 \\
    1-\alpha, & \text{if } h(1-\alpha) < 0
\end{array}
\right\}, \M_\alpha = \{z, \alpha<z<1-\alpha\},
 \text{  else, }   \P_\alpha=\M_\alpha=\emptyset. 
\end{eqnarray*}
Thus, by  Theorem~\ref{thrm_alphaRNE_multiple}, we have the following characterization of the set of $\alpha$-RNEs.
\begin{enumerate}\rm
    \item[(i)] If $\alpha > \nicefrac{1}{2}$,  the set of $\alpha$-RNEs coincides with  that of MFE-NEs, $\NE = \NEc$.
    \item[(ii)] If $\alpha \le \nicefrac{1}{2}$, then
    $
  \NE \subseteq \left (\NEc  \setminus M_\alpha\right )    \cup P_\alpha.
    $ Moreover, for the converse, we have:
    \begin{enumerate}\rm
        \item $\NEc \setminus M_\alpha \subseteq \NE$,
        \item $1-\alpha \in \NE$ if and only if $h(1-\alpha) \leq 0$,
        \item For $\alpha < \nicefrac{1}{2}$, $\alpha \in \NE$ if and only if $h(\alpha) \geq 0$.
    \end{enumerate}
\end{enumerate}}

We now derive the set of $\alpha$-RNEs, using Theorem \ref{thrm_alphaRNE_multiple}:
\begin{corollary}\label{cor_thm_alpha_RNE}
Let $|\A| = n = 2$ and  $z:= \mu_1 $. Then the set of $\alpha$-RNEs:
\begin{itemize}
    \item[(i)]  $\NE = \NEc$, i.e., coincides with that of MFG-NEs  if $\alpha > \nicefrac{1}{2}$;
    \item[(ii)]    $ \NE =  \left(\NEc \setminus  \{z: \alpha<z<1-\alpha\} \right) \cup \P_\alpha$   if $\alpha \le \nicefrac{1}{2}$,  where,
    \begin{eqnarray*}
     \P_\alpha = \{ z :   z = \alpha \mbox{ and } h(\alpha) \ge 0    \} \cup \{ z :   z = 1-\alpha \mbox{ and } h(1-\alpha) \le  0    \}. \hfill \hspace{1cm} \square  \end{eqnarray*} 
\end{itemize}
\end{corollary}
Thus, from $P_\alpha$,  the game influenced by herding crowd with two actions can admit new equilibria in the form  of $\mu = (\alpha, 1-\alpha)$ and/or $\mu= (1-\alpha, \alpha)$, depending upon the parameters.  Any classical equilibria with first component $\mu_1^*$ in the open interval $]\alpha, 1-\alpha[$ are no more equilibria when majority consists of herding crowd. 
Further, the above set of $\alpha$-RNEs aligns exactly with the set characterized in \cite[Theorem 2]{agarwal2024balancingrationalitysocialinfluence}. 


We next discuss two well-known `metrics'---which traditionally provide insights on how the outcome of the game gets affected due to the rational choices of the players (when $\alpha = 1$)---and extend them to the case with rational and herding players.

\subsection*{Price of Anarchy and Stability}
 The  Price of Anarchy  (PoA) is a `metric' in game theory  that quantifies the `loss' in some `social utility' (often defined as the sum of utilities of all the agents) due to rational choices of the agents.
 Formally, it is defined as (\cite{koutsoupias1999worst}):
 \begin{equation}\label{eqn_def_PoA}
  \text{PoA} :=\frac{ \inf \left \{   \sutil(  \mu^{NE}) :   \mu^{\text{NE}} \in \NEc  \right \} }{\us_*  },    
  \end{equation}
 where the social and social optimal utilities are defined respectively as, 
\begin{equation}
 \sutil (\mu) := \sum_{i \in \A} \mu_i u (i, \mu), \mbox{ and }
  \us_* := \sup\left\{\sutil(\mu): \mu_i \in [0,1], \ \sum_{i = 1}^{n}\mu_i = 1\right\}.  \label{eqn_social_multiple}
\end{equation}
Basically, PoA compares the `worst' NE-social-utility (the one in the numerator of \eqref{eqn_def_PoA}) with the  `best' social utility (the one in the denominator of \eqref{eqn_def_PoA}).  
 There are quite some reported instances in literature \cite{roughgarden2005selfish}, \cite{johari2004efficiency} where this PoA is significantly small (far away from one), and is well attributed to the rational behavior. One of the most intriguing questions is whether and how the existence of the herding crowd (or the reduction of the rational choices) impacts the PoA? and, to what extent? 

 We similarly have Price of Stability (PoS) that compares the `best' NE-social-utility   to the best social utility:
\begin{eqnarray}\label{eqn_def_PoS}
    \text{PoS} :=\frac{ \sup \left \{  \sutil(  \mu^{NE}) :   \mu^{\text{NE}} \in \NEc  \right \} }{ \us_*}.
 \end{eqnarray}
 Once again, it is interesting to investigate as to how the existence of herding crowd influences PoS. Before we proceed further,  we would like to comment on the differences between the minimization   and maximization  problems with regard to PoA/PoS. 
 In this paper, we maximize  the  negative of the cost functions for minimization problems. Thus  \textit{the prices (PoA/PoS) are  greater than  (or equal to) one for minimization problems, while  the same are less than (or equal to) one for the maximization problems}. In either case, the closer the prices  are to one, the smaller is the price paid due to rational (or sub-optimal) choices.
 
In the coming, we will investigate the impact of the herding crowd on PoA and PoS.
 We begin with a related study, the impact on the utilities of the rational and the irrational agents separately, and their comparison to the social optimal utility.  
 
\subsection{Comparison of utilities}
Let   $ \mu_{\alpha}^*$ be any $\alpha$-RNE at some $\alpha < 1$ and define:
\begin{align}\label{eqn_util_rational_irr_def1}
    u^R(\mu_{\alpha}^*) &:=  u(i,\mu_{\alpha}^*)   \text{ for any } i \in \A_u^R(\mu_\alpha^*), \text{ and}\\
    u^I(\mu_\alpha^*) &:=  u(i,\mu_\alpha^*) \text { for }i =\A^H(\mu_\alpha^*).\label{eqn_util_rational_irr_def2}
\end{align}
The above two quantities respectively represent the expected utilities of a typical rational and a typical herding player at any $\alpha$-RNE (see \eqref{eqn_NE_rational}-\eqref{eqn_NE_irrational}); observe here, by Definition \ref{defn_alphaRNE}(i), rational utility $u^R$ definition does not depend upon the particular choice of $i \in \A_u^R(\mu_\alpha^*)$.
The following result that compares various utilities  at various equilibria and the social optimal utility $\us_*$ \eqref{eqn_social_multiple}. 
\begin{theorem}\label{thm_util_comp} 
For any $\alpha \in (0,1)$
and any  $\mu^*_\alpha \in \NE$, we have:
\begin{itemize}
    \item[(i)]    $u^I(\mu^*_\alpha) \leq u^R(\mu^*_\alpha)$ and $u^I(\mu^*_\alpha) \leq \us_*$, and
    \item[(ii)]  $\us_* \geq u^R(\mu^*_\alpha) =  u^I(\mu_\alpha^*)$ 
    if $\mu^*_\alpha  \in \NEc$.
\end{itemize}
\end{theorem}
\begin{proof}: The proof is provided in Appendix \ref{appendix}.
\end{proof}
\noindent Part (i) of the above result shows that rational players always obtain a utility that is equal to or greater than that of the herding players. Furthermore, herding players can never achieve a utility exceeding the social optimal utility, $\us_*$. However, we will see in the coming sections, their utility can approach $\us_*$, as the rational choices (or $\alpha$) decrease.

Part (ii) states that, if the overall population measure of an  $\alpha$-RNE  also forms a classical MFG-NE, i.e., if    $ \mu_\alpha^* \in \NEc$ as well, then the rational and the irrational players attain equal utility. Further, since  $ \mu_\alpha^* \in \NEc$ one can also view it as an NE of the classical game, with all rational players --- in other words, herding players (even without making rational choices) are attaining an utility that equals the utility of a rational player at one of the classical equilibria. 
Furthermore, no player (not even a rational one) can achieve utility exceeding the social optimal utility at such equilibria. 

The PoA and PoS metrics were originally coined to study the degradation in social utility because of the rational decisions;  one can   extend the  definitions in \eqref{eqn_def_PoA}-\eqref{eqn_def_PoS} as below, now to study   the  degradation   due to herding and rational decisions:
\begin{eqnarray}\label{eqn_def_PoA_PoS_alpa}
    \text{PoA} :=\frac{ \inf \left \{  \sutil( \mu) :   \mu \in \NE  \right \} }{ \us_*  }  \ \text{ and } \  \text{PoS} :=\frac{ \sup \left \{  \sutil(   \mu) :   \mu \in \NE  \right \} }{ \us_*  }.
 \end{eqnarray}
 Finally, the results of Theorems \ref{thrm_alphaRNE_multiple}-\ref{thm_util_comp} imply the following:
 \begin{itemize}
\item[(i)] 
when the system has a significant number of rational agents with
$\alpha > 1- \nicefrac{1}{n}$, then the herding crowd does not have an impact --- there is no change in PoA or PoS because $  {\cal N}_\alpha = {\cal N}_1$, and
 
\item[(ii)] with  $\alpha \leq 1- \nicefrac{1}{n}$,  both PoA and PoS can improve compared to the classical scenario ($\alpha=1$) --- this can happen when $\M_\alpha = \emptyset$ and  $P_\alpha \neq \emptyset$. 
\end{itemize}

From Lemma \ref{Lemma_Halpha} of Appendix \ref{appendix}, every action  is an equilibrium-herding choice (i.e., $\HA = \A$),  
when $\alpha \leq 1- \nicefrac{1}{n}$;  basically,  for any $a \in \A$, there exists an $\alpha$-RNE $\mu^*_\alpha$ with $\mu^*_{\alpha, a} \ge (1-\alpha)$  (see  the proof in Appendix \ref{appendix}). Thus, there is a good chance that   
    $P_\alpha$ in \eqref{eqn_set_P_alpha} is non-empty. In fact, for both the applications considered in the next section, we observe that $\M_\alpha = \emptyset$ and $\P_\alpha \ne \emptyset$ with a  significant improvement in   PoS, which further increases with $\alpha \to 0$.  \\




We now turn our attention to some important applications, where we derive the set of $\alpha$-RNEs and compare the utilities of players across different scenarios. This will illustrate the practical implications and nuances of $\alpha$-RNE in various 
contexts.

\section{Some important games}\label{sec_exampl
es}
\subsection{Braess paradox and herding} \label{subsec_Braess_paradox}
It has been illustrated previously in the literature that, in certain transportation networks, adding a new link can increase time delays for travelers, instead of improving the traffic flow; this is contrary to the intuition. This observation was first made by Arthur Pigou in 1920 (see the monograph \cite{pigou2017welfare} and the references therein), and later named after mathematician Dietrich Braess (see the translated text in \cite{braess2005paradox}) who first established it theoretically in 1968. Subsequently Braess paradox is established in several important transport networks, and is attributed majorly to the rationality of the decision-makers (\cite{braess2005paradox}, \cite{murchland1970braess}).

The herding behavior is predominant in traffic networks (more precisely, people tend to follow their predecessors),  and hence we will investigate if such a paradox can also be observed in the presence of herding players (travelers). To this end, we first describe a pair of simple networks that exhibit the paradox. We basically consider one network with two routes,  another with an additional route, and compare the two sets of equilibria ($\alpha$-RNEs).

\subsubsection{Transportation network with two routes}
\begin{figure}[ht]
    \centering
    \includegraphics[trim = {2cm 9.5cm 7cm 9.9cm}, clip, scale = 0.15]{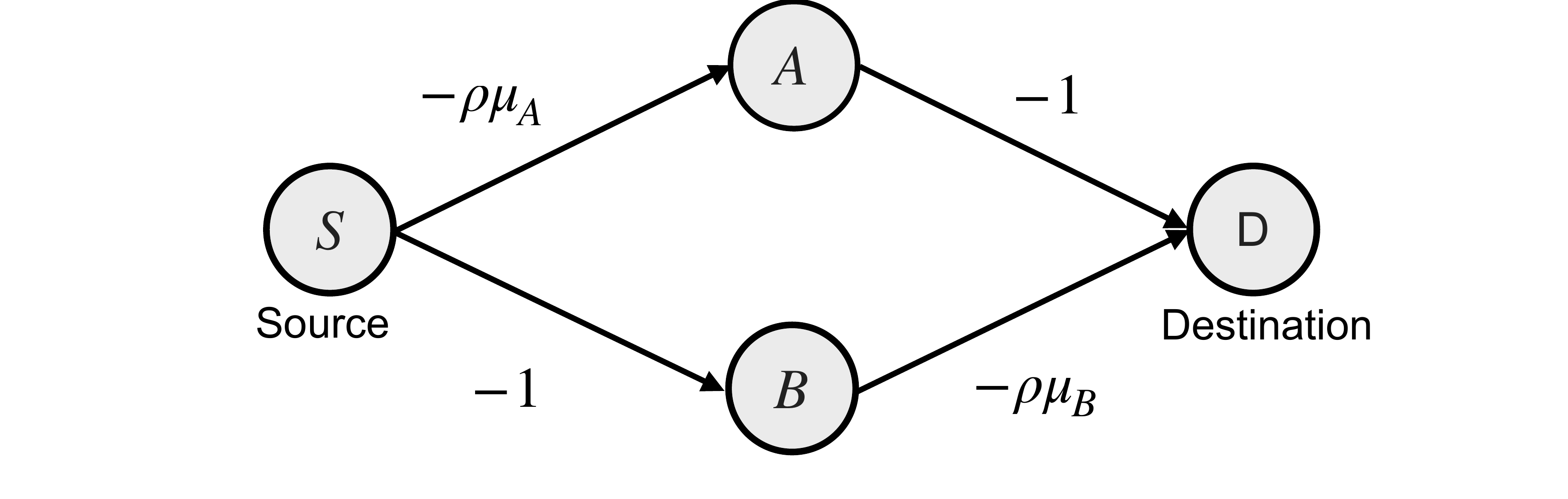}
    \caption{Transportation network with two routes}
    \label{Case_with_two_routes}
\end{figure}
Consider a transportation network with two routes from a source ($S$) to a destination ($D$). The first route passes through hub $A$, while the second is via hub $B$, as shown in Figure \ref{Case_with_two_routes}. Several travelers (players) choose one of the two routes to travel from the source to the destination. The links $SB$ and $AD$ are superior, hence the travel delays are negligible regardless of the number of commuters. However, the players utilizing these links need to pay a fixed toll cost of one unit. On the other hand,  the travel cost on the links $SA$ and $BD$ is influenced by the congestion levels and contributes towards the cost of the commuters utilizing them. To be precise, the travel costs on these links are given by $\rho \mu_A$ and $\rho \mu_B$ respectively; here  $\mu_A$ and $\mu_B$ are the fraction of players using the routes through the hubs $A$ and $B$ respectively, and $\rho \in (0,1)$ denotes the congestion coefficient. Thus, the players are involved in a noncooperative mean-field game with the action set $\A = \{A,B\}$ and the utility function:
    \begin{equation} \label{eq_braess_two_links_utility}
       u(i,\mu) = (-1-\rho \mu_A)1_{\{i = A\}} + (-1-\rho \mu_B)1_{\{i = B\}}.
    \end{equation}
Now, using Corollary \ref{cor_thm_alpha_RNE}, we immediately have:
\begin{corollary}\label{cor_two_act_rnes}
Consider the noncooperative mean-field game with the utility function given in \eqref{eq_braess_two_links_utility}. Then, the set of $\alpha$-RNEs  is given~by: 
\begin{itemize}
\item[(i)] if $\alpha \in (\nicefrac{1}{2},1]$, then $\NE = \NEc = \left\{\left(\nicefrac{1}{2},\nicefrac{1}{2}\right)\right\}$,
\item[(ii)] if $\alpha \in (0,\nicefrac{1}{2}]$, then $\NE = \{(\alpha,1-\alpha),(1-\alpha,\alpha)\}$.
\end{itemize}
\end{corollary}
Thus the classical NE is unique and equals $(\nicefrac{1}{2},\nicefrac{1}{2})$ ---  the population distributes equally across the two routes. However, with  smaller proportion of rational players  ($\alpha \leq \nicefrac{1}{2}$), two new equilibria $(\alpha, 1 - \alpha)$ and $(1 - \alpha, \alpha)$ emerge --- here all the rational players choose one route, and   the herding players opt for the other.  

Before proceeding further, 
we    study PoS and PoA of this two-route network  in the presence of  the herding crowd. 

\paragraph*{\underline{Social utility, PoA and PoS}} 
The social utility in \eqref{eqn_social_multiple}  for the  two-route network at any population measure $\mu$, using \eqref{eq_braess_two_links_utility},  is given by:
\begin{eqnarray}\label{social_util_two_network}
    \sutil (\mu ) = -1 - \rho ( (\mu_A)^2 + (\mu_B)^2 ).
\end{eqnarray}
\hide{Thus, the social utility at various $\alpha$-RNE $\mu^*_\alpha$ given in Corollary \ref{cor_two_act_rnes} is given by:
\begin{eqnarray}
    u^s (\mu^*_\alpha) = \left \{ \begin{array}{lll}
    -1 - \nicefrac{\rho}{2}     &  \mbox{  if } \alpha > \nicefrac{1}{2}\\
     -1 - \rho (\alpha^2 + (1-\alpha)^2)     &  \mbox{ if } \alpha \le \nicefrac{1}{2} \mbox{ and for both } \mu^*_\alpha \in \{ (\alpha, 1-\alpha), (\alpha, 1-\alpha) \} 
    \end{array} \right .
\end{eqnarray}
}
Thus with $\alpha > \nicefrac{1}{2}$, the $\alpha$-rational-social-utility, i.e.,  at  the unique  $\alpha$-RNE, $\mu_\alpha^* = (\nicefrac{1}{2}, \nicefrac{1}{2} )$, equals, $\sutil(\mu_\alpha^*) = -1 - \nicefrac{\rho}{2}$, see Corollary \ref{cor_two_act_rnes} . On the other hand, for $\alpha \le \nicefrac{1}{2}$, we have two distinct $\alpha$-RNEs, however with equal social utility:
$$
u^s ( (\alpha, 1- \alpha)) = u^s ( (1-\alpha,  \alpha)) =  -1 - \rho (\alpha^2 + (1-\alpha)^2).$$
Further from \eqref{social_util_two_network} and \eqref{eqn_social_multiple}, the social-optimal utility  $\us_*=-1-\nicefrac{\rho}{2}$. Thus,    PoA=PoS=1  for the classical scenario without herding population and these values remain unchanged as long as the herding crowd is  not more than $50 \%$ (i.e., for $\alpha > 1/2$), see \eqref{eqn_def_PoA_PoS_alpa}. This implies \textit{the rational decisions are also socially optimal for the two-route network and further that  the  inclusion of the herding players does not alter the system efficiency
as long as $\alpha > \nicefrac{1}{2}$}.

However, once the fraction of 
herding players exceeds $\nicefrac{1}{2}$, the PoA still remains  equal to PoS  (see Corollary \ref{cor_two_act_rnes}); we denote this common value by $p_{AS}$, i.e., let  $p_{AS}=\text{PoA}=\text{PoS}$. Unlike the previous regime, $p_{AS}$ now begins to increase from $1$ and hence degrades (recall, for minimization problems,  the gap $(p_{AS} - 1)$  illustrates the extent of degradation due to strategic decisions). In fact as $\alpha \downarrow 0$, the $p_{AS} \uparrow  1 +  \nicefrac{\rho}{(2+\rho)}$. Thus, in the transportation network with two routes, \textit{the presence of a significant herding crowd  (with $\alpha \leq \nicefrac{1}{2}$) degrades the system performance, for any $\rho$}, whereas the outcomes are efficient   with $\alpha > \nicefrac{1}{2}$. 
\subsubsection{Transportation network with an additional route} \label{subsubsec_Braess_additional_link}
To analyze the impact of an additional route, we consider an extended transportation network, as illustrated in Figure~\ref{Case with three routes}, where a direct link from hub $A$ to hub $B$ is added. Thus, each player can now choose one among three available routes from source $S$ to destination $D$: via route $A$, route $B$, or the new route $AB$, hence now the action set $\A = \{A,B,AB\}$. The corresponding utility function is:
\begin{equation} \label{eq_utility_three_links}
    u(i, \mu) = 
    \left\{
    \begin{array}{ll}
        -1 - \rho(\mu_A+\mu_{AB}), & \text{if } i = A, \\
        -1 - \rho(\mu_B+\mu_{AB}), & \text{if } i = B, \\ 
        -\rho(\mu_A + \mu_{AB})-\rho(\mu_B+\mu_{AB}), & \text{ if } i = AB,
    \end{array}
    \right.
\end{equation}
where $\mu_A, \mu_B, \text{ and }\mu_{AB}$ are the fractions of players using routes $A, B, \text{ and } AB$, respectively. Note that the players utilizing the route $AB$ incur no additional cost.
\begin{figure}[ht]
    \centering
    \includegraphics[trim = {2cm 10cm 6cm 9.9cm}, clip, scale = 0.15]{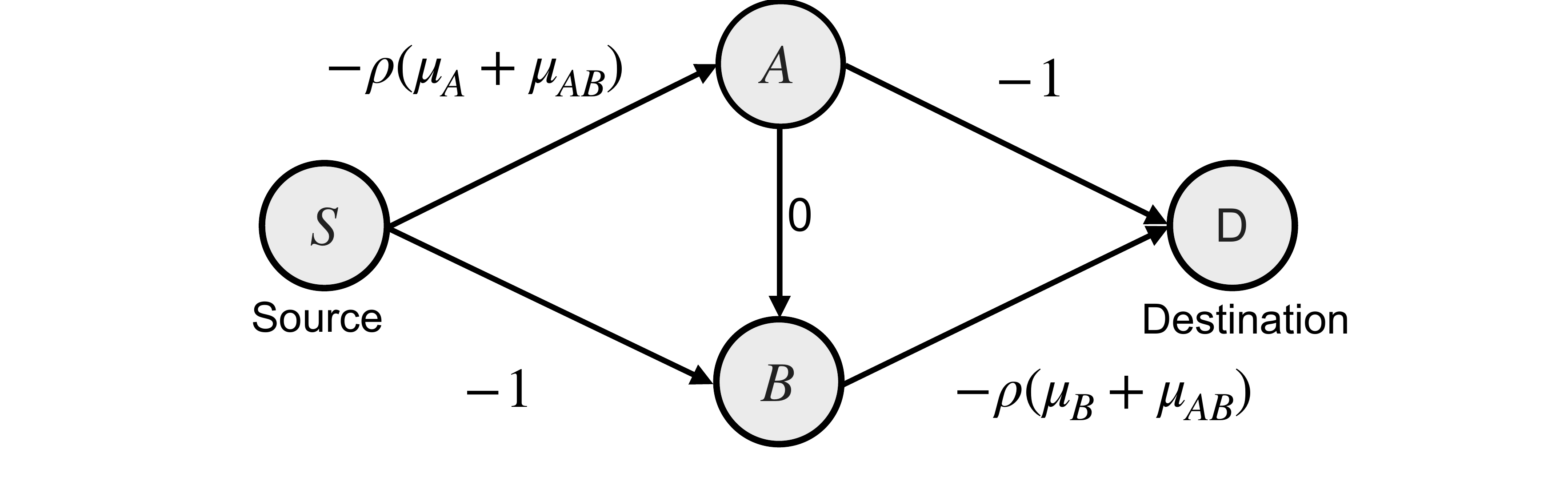}
    \caption{Transportation network with an additional route}
    \label{Case with three routes}
\end{figure}  

\noindent Then, using Theorem~\ref{thrm_alphaRNE_multiple}, we have the following result:
\begin{corollary} \label{corollary_braess_2}
Consider the noncooperative mean-field game with the utility function \eqref{eq_utility_three_links}. Then, the set of $\alpha$-RNEs is given by:
    \begin{itemize}
        \item[(i)] if $ \alpha \in (\nicefrac{2}{3},1 ]$, then $\NE = \NEc = \{\left(0,0,1\right)\}$,
        \item[(ii)] if $ \alpha \in (0,\nicefrac{2}{3}]$, then  $\NE = \NEc \cup \{\left(1-\alpha,0,\alpha\right),\left(0,1-\alpha,\alpha\right)\}$.
    \end{itemize}
\end{corollary}
Observe that, due to the addition of the new link $AB$, in the classical equilibrium no players choose routes $A$ or $B$; instead, all players end up selecting route $AB$ (whereas with only two routes $A$ and $B$, players were equally divided between them; see Corollary~\ref{cor_two_act_rnes}).
 Also, the set of NEs coincides with the set of $\alpha$-RNEs when the proportion of rational players is sufficiently large ($\alpha > \nicefrac{2}{3}$). However, when the fraction of rational players is smaller ($\alpha \leq \nicefrac{2}{3}$), two new equilibria  $(1 - \alpha, 0, \alpha)$ and $(0, 1-\alpha, \alpha)$ arise  ---
all the rational players choose the route $AB$, while all the herding players (together) opt    for route $A$ or  route $B$ at such equilibria. 

It is worth noting here that $A$ and $B$ are not equilibrium-herding choices  and $\mathcal{H}_\alpha = \{AB\}$ for scenarios with  large rational population (to be  precise, with~$\alpha > \nicefrac{2}{3}$). Furthermore,  route $AB$ emerges as a dominant strategy for rational players ---  we discuss it in detail in the immediate next, by   slightly digressing from the main topic. 

It is worth noting that $A$ and $B$ are not equilibrium-herding choices, and that
$\mathcal{H}_\alpha = \{AB\}$ when the rational population is sufficiently
large (specifically, for $\alpha > \nicefrac{2}{3}$). In addition, route $AB$
emerges as a dominant strategy for rational players; we examine this phenomenon
in detail in the following section, after a brief digression.

\paragraph*{\underline{Domination of action $AB$ with herding:}} The rational players are assumed to have common knowledge about the game. In particular, they can estimate the set of all possible equilibrium-herding choices, $\HA$ as defined in~\eqref{potential_action_herding}; recall $\HA$ is the set of actions that can become herding choice at some $\alpha$-RNE. In such scenarios, as already explained in Section \ref{subsection_interpretation}, they consider an action $i$ to be dominated by an action $j$ if, for all possible equilibrium-herding choices $\iH \in \HA$ and for all possible rational empirical-measures $\mu^R$, the following condition holds (recall any $\mu \in \mathcal{E}(\iH)$ equals $\mu = \alpha \mu^R  + (1-\alpha) 1_{\iH}$  for some $\mu^R$, see \eqref{Eqn_mathcalE}-\eqref{eqn_str_dom_eqn}):
$$
u(i,\mu) < u(j,\mu),  \mbox{for all } j \in \A\setminus\{i\}, \mbox{ with } \mu := \alpha \mu^R  + (1-\alpha) 1_{\iH}.
$$ 
For $\alpha \leq \nicefrac{2}{3}$,  $\HA= \{A, B, AB\}$ and then  \textit{$AB$ is a dominant strategy, even  after considering herding}:  (i) for any $\mu^R$ and $\mu = \alpha \mu^R+(1-\alpha) 1_{ A }$ with $A$ as the equilibrium-herding choice, the following\footnote{Clearly, $
 u(AB,\mu) =  -\rho(1-\alpha+\alpha(\mu_A^R+\mu_{AB}^R))-\rho\alpha(\mu_B^R+\mu_{AB}^R)  >  -\rho\alpha(\mu_B^R+ \mu_{AB}^R) -1    =u(B,\mu) $ and $
      u(AB,\mu) 
       > -\rho(1-\alpha+\alpha(\mu_A^R+\mu_{AB}^R)) -1 = u(A,\mu).
$}holds (as $\rho < 1$ and see \eqref{eq_utility_three_links}),
\begin{eqnarray}
  u(AB,\mu) > u(A,\mu), \text{ and } u(AB,\mu) > u(B,\mu); \label{Eqn_required_bounds}  
\end{eqnarray}
(ii) similar result holds even for the equilibrium-herding choice $B$ and $AB$, i.e.,  \eqref{Eqn_required_bounds} is true even when 
$
\mu = \alpha \mu^R+(1-\alpha) 1_{ B } \mbox{ or } \mu = \alpha \mu^R+(1-\alpha) 1_{ AB } .
$

We next study    PoS and PoA of the three-route network, before proceeding to  the comparison of the two networks.


\paragraph*{\underline{Social utility, PoA and PoS}} 
The social utility in \eqref{eqn_social_multiple} for the three-route network at any population measure $\mu$, using \eqref{eq_utility_three_links}, is given by:
\begin{eqnarray}\label{eqn_social_util_3_links_1}
    \us(\mu) = -\mu_A-\mu_B-\rho(\mu_A+\mu_{AB})^2-\rho(\mu_B+\mu_{AB})^2.
\end{eqnarray}
Thus with $\alpha > \nicefrac{2}{3}$, the social utility at the unique $\alpha$-RNE $\mu_\alpha^* = (0,0,1)$, equals   $u^s(\mu_\alpha^*) = -2\rho$, see Corollary \ref{corollary_braess_2}. On the other hand, for $\alpha \leq \nicefrac{2}{3}$ the game admits multiple $\alpha$-RNEs: in addition to $(0,0,1)$, two new equilibria emerge with equal social utility:
\begin{equation}\label{eqn_socil_util_eql_multi_rnes}
    u^s((1-\alpha,0,\alpha)) = u^s((0,1-\alpha,\alpha)) = -\rho\alpha^2+\alpha-\rho-1.
\end{equation}
Further from~\eqref{eqn_social_util_3_links_1}, the social optimal utility in~\eqref{eqn_social_multiple} differs across two regimes of~$\rho$:
\begin{eqnarray}\label{eqn_social_opt_poA_2}
    \us_* = -2\rho 1_{\{\rho \leq \nicefrac{1}{2}\}} + \nicefrac{(1-4\rho)}{(2\rho)} 1_{\{\rho > \nicefrac{1}{2}\}};
\end{eqnarray}
which naturally divides the analysis of PoA and PoS into two distinct regimes:

(i) If $\rho \leq \nicefrac{1}{2}$, we obtain $\text{PoA} = \text{PoS} = 1$ at $\alpha=1$, implying rational choices themselves are socially optimal in the classical scenario. 
The system remains efficient (the $\alpha$-RNE is socially optimal)  even in the presence of herding-crowd as long as the rationals are sufficiently high,   more precisely for all $\alpha > \nicefrac{2}{3}$    (by Corollary~\ref{corollary_braess_2},  $\NE = \NEc$). However, when   herding is prominent with  $\alpha \leq \nicefrac{2}{3}$, the efficiency reduces  --- now from \eqref{eqn_socil_util_eql_multi_rnes}-\eqref{eqn_social_opt_poA_2}, $\text{PoA} = \nicefrac{(\rho\alpha^2-\alpha+\rho+1)}{(2\rho)}$ and $\text{PoS}$ equals one;   since the  $\text{PoS}$ is still one, the system has a possibility to operate at social optimal point by appropriate influence design, as will be discussed in Section \ref{sec_mechanism_design}. Nonetheless, the PoA increases with the herding fraction --- as $\alpha \downarrow 0$ PoA $\uparrow 1+\nicefrac{(1-\rho)}{(2\rho)}$.

(ii) If $\rho > \nicefrac{1}{2}$,   the rational choices are no longer socially optimal  --- in fact with $\alpha > \nicefrac{2}{3}$ (that includes classical scenario), we have unique $\alpha$-RNE and the corresponding  $\text{PoA} = \text{PoS} = \nicefrac{(4\rho^2)}{(4\rho-1)}$, which is strictly greater than one. 
In other words,  the rational decisions degrade the system  efficiency  with  $\rho > \nicefrac{1}{2}$, i.e., when the congestion is significant. 

However, when the herding crowd is  bigger with $\alpha \leq  \nicefrac{2}{3}$, the system efficiency can further degrade.  
The social utility at $\alpha$-RNEs equals either  $u^s = - 2\rho $  (at $(0, 0, 1)$) or 
$u^s = -\rho\alpha^2+\alpha-\rho-1$  (at $(1-\alpha, 0, \alpha )$ and $(0, 1-\alpha, \alpha))$. By using simple concavity based arguments, one can show that $- 2\rho  \ge -\rho\alpha^2+\alpha-\rho-1$  iff $\alpha \leq \bar \alpha   := \nicefrac{(1-\rho)}{\rho}$. In all, PoA and PoS metrics are given by (see \eqref{eqn_def_PoA_PoS_alpa}, \eqref{eqn_social_opt_poA_2}):
$$
(\text{PoA}, \ \text{PoS}) = \left \{\begin{array}{llll}
    \vspace{2mm}
        \left(\frac{2\rho(\rho\alpha^2-\alpha+\rho+1)}{4\rho-1},   \frac{4\rho^2}{4\rho-1}\right),  & \text{ if } \alpha \in (0,\min\{\bar \alpha,\nicefrac{2}{3}\}],\\ 
        \left(\frac{4\rho^2}{4\rho-1},   \frac{2\rho(\rho\alpha^2-\alpha+\rho+1)}{4\rho-1}\right), & \text{ if } \alpha \in (\bar \alpha,\nicefrac{2}{3}].    
    \end{array} \right . 
     $$
As before, the PoA can further degrade, in fact as $\alpha \downarrow 0$ PoA $\uparrow \nicefrac{2\rho(1+\rho)}{(4\rho-1)}$, but now when $\alpha \leq \bar \alpha$; and the system efficiency remains the same as that with only rationals for any $\alpha > \bar \alpha$. 

Thus \textit{the presence of non-rational decisions has further reduced the system efficiency in both 2-link network (both PoA and PoS degrade) as well as in 3-link network (only  PoA degrades)}. However, this is not always the case, as seen  in the next example of Subsection \ref{sec_bandwidth_mutliple}. 
We now compare  the two networks.
\hide{
\subsubsection*{Utility comparison with and without herding in two routes network}
Observe from Corollary \ref{cor_two_act_rnes} that the social utility (equals $-1-\nicefrac{\rho}{2}$, as defined in \eqref{eqn_social_multiple}) remains unchanged as long as more than half of the players in the game 
are rational ($\alpha > \nicefrac{1}{2}$). However, once the proportion of 
herding players exceeds $50\%$,  the social utility becomes 
$-1-\rho+2\rho\alpha-2\rho\alpha^2$ which is lesser than $-1-\nicefrac{\rho}{2}$, i.e., the presence of 
herding players reduces the social utility whenever $\alpha \leq \nicefrac{1}{2}$, 
regardless of the $\rho$ value.

\subsubsection*{Utility comparison with and without herding in additional route network}
In this case as well, when there is a large fraction of rational players ($\alpha > \nicefrac{2}{3}$) in the game, the social utility equals $-2\rho$ which remains unchanged in the presence of herding players (see Corollary \ref{corollary_braess_2}). However, when the fraction of 
rational players is smaller ($\alpha \in (0,\tfrac{2}{3}]$), two scenarios 
arise depending on $\alpha$ and $\rho$ --  specifically, if 
$\alpha \in [\nicefrac{(1-\rho)}{\rho}, \nicefrac{2}{3}]$ and 
$\rho \in [\nicefrac{3}{5},1)$, then the presence of herding players 
improves the social utility which equals $-1-\rho+2\rho\alpha-2\rho\alpha^2$ and in all other cases, the herding players have no effect on the social utility, which remains equal to the classical case $\alpha =1$.}
\subsubsection{Paradox with additional link}
 
We finally examine the  possibility of the Braess paradox 
(discussed at the beginning of Subsection \ref{subsec_Braess_paradox})
in the transportation network considered above; we begin with  classical case ($\alpha =1$) and then study the same in the presence of the herding behavior.

At $\alpha=1$ from Corollary \ref{cor_two_act_rnes}, the social utility \eqref{social_util_two_network}  at the unique $\alpha$-RNE of the two-route network is given by $(-1-\nicefrac{\rho}{2})$, while with the addition of a new link, the same  equals $(-2\rho)$, by Corollary \ref{corollary_braess_2} and from \eqref{eqn_social_util_3_links_1}. 
Thus the  
  \textit{new link induces  equilibrium with degraded social performance   iff $(-2\rho) < (-1-\nicefrac{\rho}{2})$ --- hence the rational decisions render the addition of new-link a disaster when congestion is significant with $\rho > \nicefrac{2}{3}$ --- this  is  the paradox. Thus we consider $\rho \in (\nicefrac{2}{3},1)$ for further analysis}. 

We now turn to the central question: does the paradox persist in presence of herding players?   From Corollary \ref{corollary_braess_2}, the outcome mirrors the classical case when $\alpha > \nicefrac{2}{3}$ (i.e., with smaller   herding crowd),  and the Braess paradox still persists. 

Next we consider 
the  scenario with a larger herding crowd, with $\alpha \leq \nicefrac{2}{3}$; here  the game admits multiple equilibria by Corollary~\ref{corollary_braess_2},  which complicates the analysis. To study this scenario, we now compare \textit{the best and worst social utility at various $\alpha$-RNEs in the three-route network with the (unique) $\alpha$-RNE-social-utility in the two-route network}. 
%
Formally, let $\sutil_{b,3}(\alpha, \rho)$ and $\sutil_{w,3}(\alpha,\rho)$ respectively denote the  best and the worst $\alpha$-RNE-social utility with  three-link network  and  
 let $\sutil_{2}(\alpha,\rho)$ represent the (unique) $\alpha$-RNE-social-utility with two-route network.
 Using these quantities, we now define the relevant `comparison  functions' 
 $g_b(\alpha,\rho) := \sutil_{b,3}(\alpha,\rho)- \sutil_{2}(\alpha, \rho)$ and $g_w(\alpha, \rho) := \sutil_{w,3}(\alpha;\rho)- \sutil_{2}(\alpha;\rho)$; the   explicit forms of these functions are provided in Appendix~\ref{appen_praess_material}.
 \hide{
Observe that the behavior of the function $g(\cdot)$ varies with the parameter $\alpha$ ---  when $\alpha \in (0,\nicefrac{(1-\rho)}{\rho})$, the function is convex and has a single root at $\alpha = \nicefrac{(1-\sqrt{3-\nicefrac{2}{\rho}})}{2}$. For $\alpha \in [\nicefrac{(1-\rho)}{\rho},\nicefrac{1}{2}]$, the function is convex and does not have any root. In the interval $\alpha \in (\nicefrac{1}{2},\nicefrac{2}{3}]$, the function becomes concave and have exactly one root at $\nicefrac{(1-\sqrt{1-2\rho^2})}{(2\rho)}$ iff $\rho \in (\nicefrac{2}{3},\nicefrac{12}{17}]$.
Finally,for $\alpha \in (\nicefrac{2}{3},1]$ the function is strictly negative throughout. A pictorial representation of this behavior of the function is given as below:}

By  concavity, one can show that $g_w (\alpha, \rho) < 0$ for all $(\alpha,\rho)$ with  $\rho > \nicefrac{2}{3}$; thus the social utility at `worst' $\alpha$-RNE is always inferior to that with two-routes. Moreover, similar concavity arguments imply that   $g_b < 0$ on the following subset of the $(\alpha,\rho)$-plane with $\rho>\nicefrac{2}{3}$, see equations \eqref{eqn_gb_fun}-\eqref{eqn_g_w_fun} of Appendix~\ref{appen_praess_material}:
\begin{eqnarray*}
   \left   \{(\alpha,\rho) :\alpha \in \left(\nicefrac{\left(1-\sqrt{3-\nicefrac{2}{\rho}}\right)}{2},\nicefrac{\left(1-\sqrt{1-2\rho^2}\right)}{(2\rho)}\right) \cup \left(\nicefrac{2}{3},1\right]  \mbox{ and } \rho \in  (\nicefrac{2}{3},\nicefrac{12}{17}]   \right  \} \\
     \bigcup \left  \{(\alpha,\rho) :   
 \alpha \in \left(\nicefrac{\left(1-\sqrt{3-\nicefrac{2}{\rho}}\right)}{2},1\right] \mbox{ and } \rho \in (\nicefrac{12}{17},1)\right \}.
\end{eqnarray*}
Thus in the above sub-regime, the efficiency of the three-link network is inferior irrespective of the equilibrium reached. 

 We further plot the sign of $g_b$ for   $(  \alpha, \rho) \in [0,1] \times [\nicefrac{2}{3}, 1)$   in Figure \ref{fig:braess_paradox1} to obtain   comparison across the entire regime --- we observe  that
$g_b < 0$ for a significant part of this $(\alpha, \rho)$ regime,  see the red region ---  in this red sub-regime,  the social utility at even  the best $\alpha$-RNE is inferior to that derived with two-routes.  
\begin{wrapfigure}{r}{0.4\linewidth}
\vspace{-0mm}
    \centering
    \includegraphics[trim = {0cm 0cm 0cm 0cm}, clip, scale = 0.075]{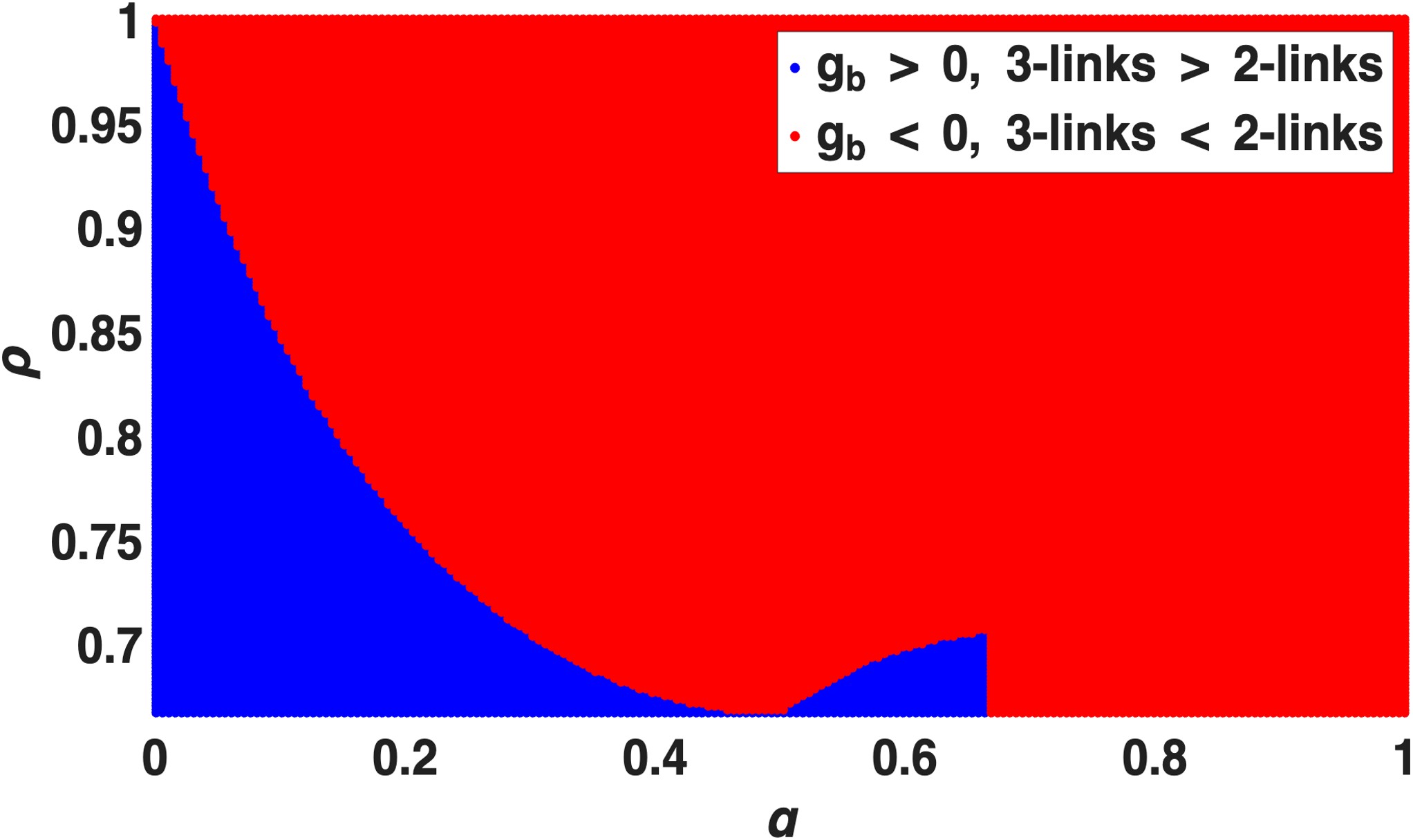}
     \caption{Comparison of  `best equilibrium' social utilities.
     } 
\label{fig:braess_paradox1}
      \vspace{-5mm}
\end{wrapfigure}
Thus, interestingly, the social utility with the three-route network improves only when: (a) an appropriate influence design can be used to steer the herding crowd towards the 'best' $\alpha$-RNE; and (b) if the fraction of rational crowd is smaller (the blue region is spread near small values of $\alpha$). 
In all other cases,   the   rational decisions render the new addition detrimental, even in the presence  of herding crowd. 

\subsection{Bandwidth sharing game with herding}\label{sec_bandwidth_mutliple}
Consider a communication network where the players share a common bandwidth to transmit their  information. 
Each player has to choose from $n$ possible transmission levels, represented by the actions from the set $\A = \left\{1,\nicefrac{1}{2},  \dots, \nicefrac{1}{n}\right\}$. 
A higher transmission level (e.g., $i=1$) allows a player to transmit at higher rate, however, it causes a greater interference to all; while a lower transmission level (e.g, $i = \nicefrac{1}{n}$) reduces the interference but also limits the transmission-rate.

The utility of a player depends on the chosen transmission level and the aggregate interference caused by all the players. The interference is determined by the weighted sum of the transmission levels of all the players and thus the utility function is given by:
\begin{equation}\label{util_bandwidth_multiple}
    u(i, \mu) = i\left(1 - c(\mu) \right), \text{ for all } i \in \A, \text{ where } c(\mu):= \sum_{j=1}^{n} \nicefrac{\mu_j}{j}.
\end{equation}
As an immediate corollary of Theorem \ref{thrm_alphaRNE_multiple}, we first obtain  the set of $\alpha$-RNEs:
\begin{corollary}\label{Bandwidth_corollary}
Consider the bandwidth-sharing game then we have the following
    \begin{itemize}
       \item[(i)] if $\alpha > 1-\nicefrac{1}{n}$ then  $\NE = \NEc =  \{(1,0,\ldots,0)\}$, and
       \item [(ii)] if  $\alpha \leq 1-\nicefrac{1}{n}$ then $\NE = \NEc \cup \B_\alpha$, where
       $$
\hspace{-6mm}
       \B_\alpha := \{\mu: \mu_1 = \alpha, \mu_i = 1-\alpha, \mbox{ for some }i \neq 1, \mbox{ and } \mu_j = 0 \ \forall  j \notin \{i, 1\}   \}. \hspace{5mm}\mbox{ \eop} $$
     \end{itemize} 
\end{corollary}
\noindent The social utility  \eqref{eqn_social_multiple} for game  \eqref{util_bandwidth_multiple}, at any population measure $\mu$ is given by:
\begin{equation}\label{eqn_soc_util_bandwidth}
    u^s(\mu) = c(\mu)\left(1-c(\mu)\right).
\end{equation}
Since the function $c \mapsto c (1-c)$ 
is concave with maximum value at $\nicefrac{1}{2}$, the social optimal utility $
u^s_*= \nicefrac{1}{4} 
$ --- observe   $c(\mu^*) = \nicefrac{1}{2}$ for $\mu^* = (0, 1, 0, \cdots, 0)$.
From Corollary \eqref{Bandwidth_corollary}, when $\alpha > 1-\nicefrac{1}{n}$, 
we have unique $\alpha$-RNE $\mu^*_\alpha = (1, 0, \cdots, 0)$ where 
both groups of players derive zero utility, 
$ u^R(\mu_\alpha^*)= u^I(\mu_\alpha^*) = 0$ (see \eqref{eqn_util_rational_irr_def1}-\eqref{eqn_util_rational_irr_def2}). 
Thus PoA=PoS=$0$ when the population is dominated by rationals   with 
$\alpha > 1-\nicefrac{1}{n}$ (we now have a maximization problem thus the prices are less than or equal to one, see comments below \eqref{eqn_def_PoS}). \textit{In other words, the rational decisions completely break the system efficiency and the scenario remains the same with smaller herding crowd}.  We next study the impact of presence of a bigger  herding population. 

Since $(1, 0, \cdots, 0)$ remains $\alpha$-RNE even with bigger herding crowd, the PoA remains at zero; however the PoS improves to
\begin{eqnarray*}
    \mbox{PoS} = \frac{ \max_{\mu \in {\cal B}_\alpha} u^s (\mu)}{ \nicefrac{1}{4}} &=&  4 \max_{j > 1}
 \left (\alpha+ \nicefrac{(1-\alpha)}{j  } \right ) \left ( 1- \left (\alpha+ \nicefrac{(1-\alpha)}{j  } \right )  \right )\\
&=& 4 \max_{j > 1}  \frac{(j\alpha+1-\alpha)(j-1)(1-\alpha)}{j^2} > 0.
\end{eqnarray*}
In fact, the PoS increases to one as $\alpha \downarrow 0$---hence, in contrast to the   example of Subsection \ref{subsec_Braess_paradox}, the presence of herding crowd improves   the system efficiency---the system in fact can operate near social optimal point when  the  rational fraction is negligible and if one can create an appropriate  influence design (see Section \ref{sec_mechanism_design})---thus sometimes ``\textit{it may be rational to be irrational!}" 


Further by optimality and the definition of equilibrium, the utility of any rational player at any $\alpha$-RNE does not exceed the social optimal utility in the classical scenario or when $\alpha=1$, i.e. $u^R(\mu^*) \leq u^s_*$ 
for all $\mu^* \in \NEc$. However, interestingly, \textit{the existence of herding crowd can strictly favor the rational players}---for example if $\alpha < \min\left\{\frac{3j-4}{4(j-1)},\frac{j-1}{j}\right\}$, then $u^R(\mu^*) > u^s_*  $ at $\alpha$-RNE, $\mu^* = \alpha \delta(1) + (1-\alpha) \delta(j)$, where $\delta(i)$ is the Dirac measure concentrated at $i$.

Thus in summary,  the existence of herding decisions can alter the outcomes in either way --- sometimes the system efficiency improves, while in some cases it can degrade --- the rationals can benefit from such irrational decisions --- surprisingly, even the herding crowd can benefit. 
However, as noticed at multiple instances   in this section, one can benefit from irrational decisions only if it is possible to appropriately influence the herding decisions. We next focus precisely on such design aspects.

\section{Mechanism design with herding players} \label{sec_mechanism_design}
There are many scenarios where a controller/moderator is required to 
design an environment leading to a game with the desired outcome. This design of `desired' games, known as mechanism design, has been greatly explored in the literature for classical games with all rational players (see e.g., \cite{narahari2014game}).
The aim of this section is to consider design for scenarios where a significant fraction of decision-makers exhibit herding behavior and when the controller has access to this knowledge. 
We aim to study this question using the rational-irrational framework developed in the previous sections. The designer can leverage upon the results of Theorem~\ref{thrm_alphaRNE_multiple} that provide the conditions for identifying $\alpha$-RNEs.

We illustrate the rational-irrational design aspects using the product selection Example \ref{example_Product}. Say there exists a manufacturer that produces the three products given in the said example. Further, say the manufacturer has to decide its production activity depending upon the anticipated outcome of the game (which mirrors the market demand). 
It has to decide the proportions of the products to be manufactured. It may also consider the profit margin of each product in its decision.  

If all customers were rational, the manufacturer could easily predict the demand --- recall, the classical game has a unique NE  $(1, 0, 0)$. Thus, the manufacturer can anticipate the product $1$ to dominate the market and plan the production accordingly. Even if it desires to produce other products, it does not see much prospects.  

Now, say a fraction of customers exhibit herding behavior, and say this fraction is at least $\nicefrac{1}{3}$. Further, assume that the manufacturer is aware of the composition. Then the situation is drastically different --- it can be a boon or a curse for the manufacturer. Recall that the game has three $\alpha$-RNEs, viz., $(1,0,0)$, $(\alpha, 1-\alpha,0)$, and $(\alpha,0,1-\alpha)$  (see Example \ref{example_Product}). In other words, the future market demand would be concentrated on product $1$ alone, or on products $1$ and $2$, or on products $1$ and $3$. Consequently, the outcome of the game and thus the market demand becomes unpredictable. To avoid losing customers, the manufacturer may have to produce all products. This can result in overproduction and associated losses.

On the other hand, the existence of herding behavior also provides the manufacturer an opportunity to produce `desirable' products. For instance, say in the same example, the manufacturer is keen on producing the product $2$, probably for higher profit margin. If the entire population were rational, it had no option. However with herding customers, it has a chance by using some additional strategies to propel the game towards the outcome $(\alpha,1-\alpha,0)$. This could be achieved (for example) through targeted strategies such as excessive advertising for product $2$, which influences the initial decisions of some customers, eventually leading to the product $2$ being chosen by all herding customers (following themselves at the limit, see Section \ref{section_newnotion}) and the emergence of $(\alpha, 1-\alpha, 0)$ as the outcome. 



In summary, only product $1$ had demand if all customers were rational. However, herding customers pave the way for creating demand for all products.  Further, by some external forces like advertisements, word-of-mouth, social media propagation, etc., one can probably create demand for a desired product. Essentially, the manufacturer can plan strategically, predicting rational and herding actions with the help of Theorem ~\ref{thrm_alphaRNE_multiple}.

Thus, if the controller can strategically create an initial influence that propels the herding players towards a `suitable' action, 
it could achieve a `desirable' outcome. It is important to note here that the controller would choose  an influence only towards the ones among ${\cal H}_\alpha$ (see \eqref{potential_action_herding})  as only these actions are sustainable (this is discussed in sub-section \ref{subsection_interpretation}). This targeted approach leverages upon the properties of $\alpha$-RNEs (which help to decide the desired outcome among the possible choices). 
This scenario can be modeled as a Stackelberg game, where the controller creates the initial influence.
We will delve deeper into this aspect in the immediate next for general games with herding players, by formally constructing the corresponding influence design game.

\subsection{Influence Design:    Stackelberg framework} \label{subsec_Influence_design}
We consider a three-level Stackelberg framework to formalize the ideas related to influence design. At the upper level, a controller or moderator aims to achieve a certain `required objective' $f$, by designing an appropriate influence, which in turn modulates the decisions of the herding players; we discuss various choices for $f$ later. 
This influence can be created with the help of influencers or prominent figures, such as celebrities, sports personalities, politicians, etc., who are capable of shaping public perception and behavior via (say) advertisements, word-of-mouth, etc. (as discussed in the product selection Example \ref{example_Product}). The influence is specified as an empirical measure $\ei$ on the set of actions ${\cal A}$, and is either the resultant of the actions chosen by the influencers themselves (e.g., as in \cite{singh2024stochastic}) or is created by the propelling forces towards a set of choices among $\HA$.

The middle level corresponds to the herding players (($1-\alpha$) fraction of the population),  whose choices are driven by the created influence $\ei$;  they simply follow the trends. Specifically, they choose the maximal action indicated by the influence measure  $\ei$:  
\begin{equation}\label{eq_action_herd_stac}
       i_*^H (\ei) = \mbox{Arg max}_{i \in {\cal A}}    \ei(i),
\end{equation}
if required tie-breaking can be considered as in \eqref{eqn_NE_irrational}.

Finally, at the lower level, we consider the rational players that constitute $\alpha$ fraction of the population.  These players choose their actions based on the trends of herding crowd, and in pursuit to optimize their own utility function. Consequently, 
the interaction among the rational players forms a noncooperative mean-field game, conditioned on the actions of the herding players $\iH$, specified by:
\begin{eqnarray}\label{eq_util_rational_stack}
    {\cal G} (\iH) : =  \left < \A,u(\cdot, \  \cdot; \ \iH)\right >,
\end{eqnarray}
where $u(i, \ \mu^R;  \ \iH)$ represents the utility function of a typical rational player choosing action $i \in {\cal A}$, if $\mu^R $ were the empirical measure of the rational players and herding players had chosen the action $\iH$.  

The empirical measure $\mu^R_*(\iH)$ is the  MFG-NE of the lower-level game for any herding action $\iH$ if it satisfies the following:
$$
\mbox{support}(\mu^R_*(\iH)) \subseteq {\rm Arg} \max_{i \in {\mathcal A}} u\left(i, \  \mu^R_*(\iH); \ \iH\right).
$$
Observe from Definition \ref{defn_alphaRNE}, when $\iH \in {\cal H}_\alpha$, then $\left ( \mu_*(\iH), \mu_*^R (\iH) \right )$ with $\mu_*(\iH) : = \alpha \mu^R_* (\iH) + (1-\alpha) \delta (\iH) $  is an $\alpha$-RNE; here    $\delta$ represents Dirac measure. 

 The objective function $f$  of the controller depends upon the actions chosen by the herding and the rational population, which in turn depend upon the influence $\ei$. Then the controller at the upper level is interested in the following influence design game, 
\begin{eqnarray}
\label{eq_initial_influence}
   \ei^* \in \mbox{Arg}\max_{\ei : \iH_*(\ei) \in \HA}f\left( \iH_*(\nu) , \ \mu^R_*(\iH_* (\nu) )\right).
\end{eqnarray}
The tuple $\left(\ei^*,  \iH_*,\mu^R_* \right)$ with $\iH_*= \iH_*(\nu^*)$ and $\mu^R_* = \mu^R_*(\iH_*(\nu^*)) $ represents the \textit{Stackelberg Equilibrium (SE)} of the influence game (when the game is well-posed). 

By virtue of the tie-breaking rule and the definition of $\iH_*(\ei)$ (see \eqref{eq_action_herd_stac}), the  optimization problem in \eqref{eq_initial_influence} can be rewritten as:
\begin{eqnarray}
\label{eq_initial_influence1}
  \max_{ \iH \in \HA}f\left( \iH , \ \mu^R_*(\iH)\right).
\end{eqnarray}
The above equivalence is true only when $f$ does not depend explicitly\footnote{ Basically we assume  $f$ depends only upon the actions of the agents and not upon the influence $\nu$ created.  If required one can extend the framework and let $f$ depend additionally on $\nu$.} on the influence $\ei$ as in \eqref{eq_initial_influence}, and because (trivially) for every $\iH$ there exists a $\ei$ such that $i_*^H(\ei) = \iH$. If  $i_*^H$ is an optimizer of \eqref{eq_initial_influence1}, then any $\ei^*$ with $i_*^H(\ei^*) = i_*^H$ becomes an optimal influence.

One can discuss SE \textit{only when the game is well-posed}, i.e., when the lower level equilibrium $\mu^R_*$ exists and is unique for each $\ei$ and $\iH_*(\ei)$. By equivalence to  \eqref{eq_initial_influence1},
 the problem is well-posed if there exists a unique MFG-NE for the lower level game  ${\cal G}(\iH)$ between the rational players, for each $\iH \in \HA$.

\noindent\textbf{Influence design for product selection game \ref{example_Product}:}
We now analyze the Example \ref{example_Product} through the lens of influence design. Recall there are three $\alpha$-RNEs with $\HA = \{1,2,3\}$, and that the manufacturer prefers product 2. Hence, the manufacturer would consider an influence design game with, 
\begin{equation}\label{eq_product_f}
    f(\iH, \mu^R) =  (1-\alpha)1_{\{\iH=2\}}  + \alpha \mu^R(2).
\end{equation}
Clearly, for any $\iH \in \HA$, the measure $\mu_*^R(\iH) = (1,0,0)$ is the unique NE for the lower level game ${\cal G}(\iH)$ among the rational players. Thus, the influence game is well-posed. From \eqref{eq_initial_influence1} and \eqref{eq_product_f}, it follows that $\iH = 2$ is the solution to the optimization problem given in \eqref{eq_initial_influence1}. Consequently, the tuple $(\nu^*,i_*^H,\mu_*^R)$ represents an SE for this game, when $i_*^H = 2$, $\mu_*^R = (1,0,0)$ and the influence $\nu^*$ is such that $\nu^*(2) > \max\{\nu^*(1),\nu^*(3)\}$. The manufacturer, for example, can achieve such an influence $\nu^*$ through excessive advertisement of product 2. If such an influence can be achieved, the eventual market demand distribution (combined market demand due to herding and rational customers) reaches $\mu_* = (\alpha,1-\alpha,0)$. In such circumstances, the manufacturer can strategically plan to produce products $1$ and $2$ in proportions that align with $\mu_*$.

{\bf Remarks:} 
There are several (summarizing) remarks regarding the influence design game and the design with herding players.

Typically, the game with herding players will have multiple $\alpha$-RNEs and 
the SE in the Subsection \ref{subsec_Influence_design} identifies the `optimal' among them for the controller. In other words, \textit{in cases of ambiguity (with more than one $\alpha$-RNE) and if it is possible to create the required influence, the controller can propel towards the most favorable outcome by creating influence $\ei^*$ (provided by the influence game).} Observe this is possible only due to the presence of herding players, who can be influenced. 

 When the controller can not create the required influence, its choices have to depend upon all possible outcomes/$\alpha$-RNEs. \textit{For instance,  it could apply the pessimal anticipation rule to predict the worst-case scenario for the future outcomes of the game} (the manufacturer has to produce all three products in the product selection game \ref{example_Product}, if its objective is not to lose the customers).

We would also like to clarify here that the focus  is only on the influence (to be created) and not on the means of creation. There are some strands of literature that focus on the latter topic. 
For instance,   the controller in \cite{singh2024stochastic} designs optimal incentives to sufficiently motivate the influencers towards vaccination, which in turn ripples positive influence on the population resulting in desired levels of disease eradication. However,
for the influence design game, we assume it is possible to create any required influence. 


This framework can be used  to capture yet another scenario---the rational players themselves can also create the empirical measure $\ei$ at the upper level, in order to lead or mislead the herding players and reap the resultant benefits. The function $f$ (see \eqref{eq_initial_influence}) in this context corresponds to the utility function $u$ (see \eqref{eq_util_rational_stack}) of the rational players. This approach yields the outcome of the Stackelberg game, which is advantageous to rational players.

\section{Conclusion}\label{sec_Conclusion}

In this paper, we develop an analytical framework to study games with a large population in which an $\alpha$-fraction of players act rationally, while the remaining $(1-\alpha)$-fraction follows the majority (herding behavior). We introduce a new equilibrium concept, called the $\alpha$-Rational Nash Equilibrium ($\alpha$-RNE), and discuss its interpretations, including prescription, prediction, and the limit of some learning dynamics.
We then show that the set of equilibria can change as the fraction of herding players increases: some classical equilibria may disappear while some new equilibria may emerge. Interestingly,  rational
players always benefit in the presence of herding  players---sometimes obtain even better than the social optimal~utility. Moreover, in some cases, the herding players may also obtain better, even almost the social optimal utility. \textit{Perhaps, it is rational to be irrational sometimes!} 

We further examine the  influence of herding behavior on system efficiency through performance measures such as Price of Anarchy (PoA). In transportation networks, 
it is well known that adding a new link reduces system efficiency (overall travel time increases) due to rational decisions; this was first identified by Pigou and later analyzed in detail by Braess. We observed that this inefficiency can be handled and the new link becomes beneficial if the herding crowd is predominant and if that crowd can be influenced amicably. 
The improvement becomes more pronounced for larger herding fractions and/or lower congestion levels. However, if herding decisions cannot be influenced, the additional link degrades the performance.
In a collision-based bandwidth sharing game, it is well known that rational decisions result in a significant PoA.
We observed that the presence of herding crowd introduces a new equilibrium which is substantially more efficient---in fact, when the rational crowd is significantly smaller, even the irrational players derive utility close to the social optimal value. 

Finally, we discuss mechanism-design aspects in the presence of herding. The larger set of equilibria can offer useful opportunities when an appropriate influence is created. However, if such an   influence cannot be created, the larger set of equilibria increases the design challenge, as undesirable outcomes may arise. We illustrate these issues through a Stackelberg framework for influence design.




\bibliographystyle{elsarticle-harv}
\bibliography{references}

@misc{vasal2020alpha,
  author = {Vasal, Deepanshu and Berry, Randall},
  title  = {Alpha-Robust Equilibrium in Anonymous Games},
  year   = {2020},
  note   = {SSRN Working Paper No. 3643821}
}

@misc{vasal2020fault,
  author = {Vasal, Deepanshu and Berry, Randall},
  title  = {Fault Tolerant Equilibria in Anonymous Games: Best Response Correspondences and Fixed Points},
  year   = {2020},
  note   = {arXiv:2005.06812}
}

@book{sundaram1996first,
  author    = {Sundaram, Rangarajan K.},
  title     = {A First Course in Optimization Theory},
  publisher = {Cambridge University Press},
  address   = {Cambridge},
  year      = {1996}
}

@inproceedings{agarwal2024balancingrationalitysocialinfluence,
  title={Balancing rationality and social influence: Alpha-rational Nash equilibrium in games with herding},
  author={Agarwal, Khushboo and Avrachenkov, Konstantin and Kavitha, Veeraruna and Vyas, Raghupati},
  booktitle={International Conference on Game Theory for Networks},
  pages={91--107},
  year={2025},
  organization={Springer}
}

@book{carmona2018probabilistic,
  author    = {Ren{\'e} Carmona and Fran{\c{c}}ois Delarue and others},
  title     = {Probabilistic Theory of Mean Field Games with Applications I-II},
  publisher = {Springer},
  year      = {2018},
  series    = {Stochastic Analysis Series},
  address   = {New York, USA}
}

@article{eliaz2002fault,
  title={Fault tolerant implementation},
  author={Eliaz, Kfir},
  journal={The Review of Economic Studies},
  volume={69},
  number={3},
  pages={589--610},
  year={2002}
}

@book{camerer2011behavioral,
  author    = {Camerer, Colin F},
  title     = {Behavioral Game Theory: Experiments in Strategic Interaction},
  publisher = {Princeton University Press},
  address   = {Princeton, NJ, USA},
  year      = {2011}
}

@article{thaler2018cashews,
  title={From cashews to nudges: The evolution of behavioral economics},
  author={Thaler, Richard H},
  journal={American Economic Review},
  volume={108},
  number={6},
  pages={1265--1287},
  year={2018},
  publisher={American Economic Association 2014 Broadway, Suite 305, Nashville, TN 37203}
}

@article{schultz2008introduction,
  author  = {Schultz, Wolfram},
  title   = {Introduction. neuroeconomics: the promise and the profit},
  journal = {Philosophical Transactions of the Royal Society B: Biological Sciences},
  volume  = {363},
  number  = {1511},
  pages   = {3767--3769},
  year    = {2008}
}

@book{sandholm2010population,
  title={Population Games and Evolutionary Dynamics},
  author={Sandholm, William H.},
  year={2010},
  publisher={MIT Press},
  address={Cambridge, MA}
}

@article{mckelvey1995quantal,
  title={Quantal response equilibria for normal form games},
  author={McKelvey, Richard D and Palfrey, Thomas R},
  journal={Games and Economic Behavior},
  volume={10},
  number={1},
  pages={6--38},
  year={1995}
}

@article{banerjee1992simple,
  author  = {Banerjee, Abhijit V.},
  title   = {A Simple Model of Herd Behavior},
  journal = {The Quarterly Journal of Economics},
  volume  = {107},
  number  = {3},
  pages   = {797--817},
  year    = {1992}
}

@article{agarwal2025two,
  title={Two choice behavioral game dynamics with myopic-rational and herding players},
  author={Agarwal, Khushboo and Avrachenkov, Konstantin and Vyas, Raghupati and Kavitha, Veeraruna},
  journal={Proceedings of the ACM on Measurement and Analysis of Computing Systems},
  volume={9},
  number={1},
  pages={1--26},
  year={2025},
  publisher={ACM New York, NY, USA}
}

@book{narahari2014game,
  title     = {Game Theory and Mechanism Design},
  author    = {Narahari, Yadati},
  year      = {2014},
  publisher = {World Scientific},
  address   = {Singapore}
}

@article{holt2004nash,
author = {Charles A. Holt  and Alvin E. Roth },
title = {The Nash equilibrium: A perspective},
journal = {Proceedings of the National Academy of Sciences},
volume = {101},
number = {12},
pages = {3999-4002},
year = {2004},
doi = {10.1073/pnas.0308738101},
URL = {https://www.pnas.org/doi/abs/10.1073/pnas.0308738101},
eprint = {https://www.pnas.org/doi/pdf/10.1073/pnas.0308738101}}

@book{neumann_morgenstern_1944,
 ISBN = {9780691119939},
 URL = {http://www.jstor.org/stable/j.ctt1r2gkx},
 author = {John von Neumann and Oskar Morgenstern and Ariel Rubinstein},
 publisher = {Princeton University Press},
 title = {Theory of Games and Economic Behavior: 60th Anniversary Commemorative Edition},
 urldate = {2026-01-28},
 year = {1944}
}

@article{nash1950equilibrium,
  title={Equilibrium points in n-person games},
  author={Nash Jr, John F},
  journal={Proceedings of the national academy of sciences},
  volume={36},
  number={1},
  pages={48--49},
  year={1950},
  publisher={National Acad Sciences}
}

@article{anderson2002logit,
  title={The logit equilibrium: A perspective on intuitive behavioral anomalies},
  author={Anderson, Simon P and Goeree, Jacob K and Holt, Charles A},
  journal={Southern Economic Journal},
  volume={69},
  number={1},
  pages={21--47},
  year={2002},
  publisher={Wiley Online Library}
}

@article{murchland1970braess,
  title={Braess's paradox of traffic flow},
  author={Murchland, John D},
  journal={Transportation Research},
  volume={4},
  number={4},
  pages={391--394},
  year={1970},
  publisher={Elsevier}
}

@article{singh2024stochastic,
  author = {Vartika Singh and Veeraruna Kavitha},
  title = {Stochastic vaccination game among influencers, leader and public},
  journal = {Dynamic Games and Applications},
  volume = {14},
  number = {5},
  pages = {1268--1316},
  year = {2024},
  publisher = {Springer US},
}

@article{morone2008simple,
  title={A simple note on herd behaviour},
  author={Morone, Andrea and Samanidou, Eleni},
  journal={Journal of Evolutionary Economics},
  volume={18},
  pages={639--646},
  year={2008},
  publisher={Springer}
}

@inproceedings{koutsoupias1999worst,
  title={Worst-case equilibria},
  author={Koutsoupias, Elias and Papadimitriou, Christos},
  booktitle={Annual Symposium on Theoretical Aspects of Computer Science},
  pages={404--413},
  year={1999},
  organization={Springer}
}

@article{johari2004efficiency,
  title={Efficiency loss in a network resource allocation game},
  author={Johari, Ramesh and Tsitsiklis, John N},
  journal={Mathematics of Operations Research},
  volume={29},
  number={3},
  pages={407--435},
  year={2004},
  publisher={INFORMS}
}

@book{roughgarden2005selfish,
  author    = {Roughgarden, Tim},
  title     = {Selfish Routing and the Price of Anarchy},
  publisher = {MIT Press},
  address   = {Cambridge, MA},
  year      = {2005}
}

@misc{eyster2009rational,
  author = {Eyster, Erik and Rabin, Matthew},
  title  = {Rational and Naive Herding},
  year   = {2009},
  note   = {CEPR Discussion Paper No. DP7351}
}

@incollection{pigou2017welfare,
  title     = {Welfare and Economic Welfare}, 
  author    = {Arthur Cecil Pigou},
  booktitle = {The Economics of Welfare},
  pages     = {3--22},
  year      = {2017},
  publisher = {Routledge},
  address   = {London}
}

@book{gigerenzer2002bounded,
  author    = {Gigerenzer, Gerd},
  title     = {Bounded Rationality: The Adaptive Toolbox},
  publisher = {MIT Press},
  address   = {Cambridge, MA},
  year      = {2002}
}

@incollection{camerer2004behavioural,
  author    = {Camerer, Colin F and Ho, Teck-Hua and Chong, Juin Kuan},
  title     = {Behavioural Game Theory: Thinking, Learning, and Teaching},
  booktitle = {Advances in Understanding Strategic Behaviour: Game Theory, Experiments and Bounded Rationality},
  editor    = {S. Huck},
  publisher = {Springer},
  address   = {Berlin, Germany},
  pages     = {120--180},
  year      = {2004}
}

@article{braess2005paradox,
  title={On a paradox of traffic planning},
  author={Braess, Dietrich and Nagurney, Anna and Wakolbinger, Tina},
  journal={Transportation Science},
  volume={39},
  number={4},
  pages={446--450},
  year={2005},
  publisher={INFORMS}
}

@article{feldman2005overcoming,
  title={Overcoming free-riding behavior in peer-to-peer systems},
  author={Feldman, Michal and Chuang, John},
  journal={ACM SIGecom Exchanges},
  volume={5},
  number={4},
  pages={41--50},
  year={2005},
  publisher={ACM New York, NY, USA}
}

@incollection{hardin2003free,
  author    = {Hardin, Russell and Cullity, Garrett},
  title     = {The Free Rider Problem},
  booktitle = {The Stanford Encyclopedia of Philosophy},
  editor    = {Edward N. Zalta},
  year      = {2003},
  publisher = {Metaphysics Research Lab, Stanford University}
}

\section{Appendix}\label{appendix}

\noindent \textbf{Proof of Theorem \ref{thrm_alphaRNE_multiple}:}
We divide the proof into two steps as follows:

\noindent \textbf{Step 1:} To prove that, for all $\alpha \in (0,1]$, $\left ((\NEc  \setminus \M_\alpha\right)    \cup \P_\alpha) \subset \NE.$

Let $\mu \in \P_\alpha$ and say $\Amu = k$ for some $k \in \mathcal{A}$. By Definition \eqref{eqn_set_P_alpha}, we have $\mu_k = 1-\alpha$. Define $\mu_k^R := 0$ and $\mu_i^R := \nicefrac{\mu_i}{\alpha}$ for all $i \neq k \in \mathcal{A}$. Observe that $\mu_i^R \in [0,1]$ for all $i \in \mathcal{A}$ ---  this is because $\sum_{i\neq k} \mu_i = \alpha$ and hence $\mu_i \leq  \alpha$ for all $i \neq k$. 
   
To begin with, observe using the above definitions that \eqref{eqn_NE_total} is satisfied with $(\mu,\mu^R)$. Now, to illustrate that the pair also satisfies \eqref{eqn_NE_rational}. Consider any  $i  \in \support(\mu^R)$, i.e., with  $\mu_i^R > 0$;  then by construction    $i \neq k = \Amu$  and also $\mu_i >0$ (as $\mu_i^R = \nicefrac{\mu_i}{\alpha}$); and hence: 
$$
i \in \left (\support(\mu) \setminus \{\Amu\} \right ) \subseteq \Au \  \  \text{ (the last inclusion is by  \eqref{eqn_set_P_alpha})}.
$$
Hence, $\support(\mu^R) \subseteq \Au$  and therefore  by Definition \ref{defn_alphaRNE}, $\mu \in \NE$.

Next, consider any $\mu \in \NEc \setminus \M_\alpha$. Since $\mu \notin \M_\alpha$,  there exist at  least one $i$ such that $\mu_i \geq 1-\alpha$. Let  $k := \Amu$, then by definition, $\mu_k \ge 1-\alpha$. 
Now, set $\mu^R$ as in \eqref{Eqn_muR_given_mu}.
Observe that $\mu_i^R \in [0,1]$ for all $i \in \mathcal{A}$ and satisfies \eqref{eqn_NE_total} as before. Further once again,  $\support(\mu^R) \subseteq \support(\mu)$ and now since $\mu \in \NEc$, we also have that $\support(\mu) \subseteq \Au.$ Thus again, by Definition \ref{defn_alphaRNE}, $\mu \in \NE$.
   
\noindent \textbf{Step 2:} To prove that, for all $\alpha \in (0,1]$, $\NE \subseteq \left((\NEc  \setminus \M_\alpha\right )\cup \P_\alpha)$.

Let $\mu \in \NE$ with $k = \Amu$. 
Then $\mu_k \ge (1-\alpha)$, and thus $\mu \notin \M_\alpha$. Define $\mu^R$ as in \eqref{Eqn_muR_given_mu}, then  by Definition \ref{defn_alphaRNE}, $\support(\mu^R)  \subseteq \Au$.


To begin with, say $k \notin \Au$, 
then $\mu_k  = 1-\alpha$ by   Definition \ref{defn_alphaRNE}.  By \eqref{Eqn_muR_given_mu}, $\mu_k^R = 0$, and so
$\support(\mu) \setminus \{k\} = \support(\mu^R)    \subseteq \Au$.
Therefore $\mu \in \P_\alpha$, see  \eqref{eqn_set_P_alpha}.

Now, let $k \in \Au$.  Then, $\mu_k \ge  1-\alpha$,  so by \eqref{Eqn_muR_given_mu}, either $ \support(\mu^R) = \support(\mu)$
or $\support(\mu^R)=\support(\mu) \setminus \{k\} $. In any case, $ \support(\mu) \subseteq \Au$. Hence, by \eqref{eqn_NE}, $\mu \in \NEc \setminus \M_\alpha$. \eop

\medskip 



\noindent \textbf{Proof of Theorem \ref{thm_util_comp}:} 
\textbf{Part (i):} Define the function $f(\mu) = \sum_{i=1}^{n} \mu_i u(i,\mu)$ for all population measures, $\mu $. Consider any $\mu^*_\alpha \in \NE$, then (see \eqref{eqn_social_multiple}):
    \begin{align*}
          \us_* &= \sup_{\mu} f(\mu) = \max\left\{\sup_{\{\mu ; \mu \neq \mu^*_\alpha\}} f(\mu), f(\mu^*_\alpha)\right\} \geq f(\mu^*_\alpha).   
    \end{align*}
    Next,  we will prove that $f(\mu^*_\alpha) \geq u_\alpha^I(\mu_\alpha^*)$ and $u_\alpha^I(\mu^*_\alpha) \leq u_\alpha^R(\mu^*_\alpha)$ for all $\mu^*_\alpha \in \NE$.

 Consider $\mu^*_\alpha \in \NE$, then we have  (see \eqref{eqn_util_rational_irr_def1}-\eqref{eqn_util_rational_irr_def2}):
 \begin{eqnarray}\label{eqn_f_mu_star_1}
     f(\mu^*_\alpha) = \sum_{i=1}^{n} \mu_{\alpha,i}^* u(i,\mu^*_\alpha) \leq \max_{i \in \A} u(i,\mu_\alpha^*) = u^R(\mu_\alpha^*).
 \end{eqnarray}
 Again, by Definition \ref{defn_alphaRNE}, we have the following:
\begin{eqnarray}\label{eqn_f_mu_star_2}
    f(\mu^*_\alpha) = \mu^*_{\alpha,a^H} u(a^H,\mu_\alpha^*)+\sum_{i \in \A \setminus \{a^H\} } \mu_{\alpha,i}^* u(i,\mu^*_\alpha) \geq u(a^H,\mu_\alpha^*) = u^I(\mu_\alpha^*).
\end{eqnarray}
 Thus, from \eqref{eqn_f_mu_star_1} and \eqref{eqn_f_mu_star_2}, we have:
 $$
 u^I(\mu_\alpha^*) \leq u^R(\mu_\alpha^*).
 $$
\textbf{Part (ii): } 
Suppose $\mu^*_\alpha \in \NE$. 
  Then we have,  
$u^I(\mu^*_\alpha) = u^R(\mu^*_\alpha)$, which by \eqref{eqn_f_mu_star_1}- \eqref{eqn_f_mu_star_2} implies   $\us_* \geq f(\mu_\alpha^*) =u^I(\mu^*_\alpha) = u^R(\mu^*_\alpha)  $.  \eop  

\hspace{1mm}
\begin{lem}
\label{Lemma_Halpha} 
    For every $\alpha \le 1 - \nicefrac{1}{|\mathcal{A}|}$, we have $\HA = \A$  ($\HA$ is  defined in \eqref{potential_action_herding}). 
\end{lem}
\noindent \textbf{Proof.}
Consider any fixed $a \in \A$. Define the following mean-field game with  $\A$ as the set of actions and the utility function as,
\begin{eqnarray*}
    w^a (i, \nu) := u(i, (1-\alpha) \delta_{a}  + \alpha \nu ) \mbox{ for all }  i \in \A \mbox{ and all  } \nu \in \P ( \A ), 
\end{eqnarray*}
where $\delta_a$ is the probability measure that concentrates at $a$.
If $\nu^*$ is any MFG-NE of $w^a$-game, then it follows directly from Definition \ref{defn_alphaRNE}, \eqref{Eqn_muR_given_mu} and the definition of $w^a$-game that $\mu^* := (1-\alpha) \delta_{a}  + \alpha \nu^* $ is an $\alpha$-RNE of the original $u$-game (see Section \eqref{section_newnotion}) with $\A^H(\mu^*) = a$, when $\alpha \le 1 - \nicefrac{1}{|\mathcal{A}|}$. Thus, $a \in \HA$, if $w^a$-game has a MFG-NE.   

Thus, it suffices to prove that $w^a$-game has MFG-NE for each $a \in \A$. 
This is immediate to see, as any finite action mean-field game has an MFG-NE --- follows by Kakutani fixed point theorem (see e.g., \cite[Theorem 9.26]{sundaram1996first}) applied to the following correspondence, where for each $\mu \in \P(\A)$, 
\begin{align*}
\Phi(\mu) &:= \text{Arg} \max_{\mu' \in {\cal D} (\mu) } w(\mu', \mu), \\
\mbox{ with } w(\mu', \mu) &:= \sum_a \mu'(a) u(a, \mu), \mbox{ and, } {\cal D}(\mu) = \P(\A); 
\end{align*}
observe $\mu \mapsto \Phi(\mu)$ is non-empty upper semi continuous (USC), compact, convex correspondence by Maximum theorem, \cite[Theorem 9.17]{sundaram1996first} --- here $\Theta = \P (\A)$, ${\cal D}(\mu) = \P(\A)$ and clearly $\mu' \mapsto w (\mu', \mu)$ is linear and hence convex and continuous.  \eop

\subsection{Material of transportation network application in Subsection \ref{subsec_Braess_paradox}}\label{appen_praess_material}
Using \eqref{eq_braess_two_links_utility}-\eqref{eq_utility_three_links} and Corollaries \ref{cor_two_act_rnes}-\ref{corollary_braess_2} and with $\bar \alpha := \nicefrac{(1-\rho)}{\rho}$, we have:
 \begin{eqnarray*}
    \sutil_{b,3}(\alpha;\rho) \hspace{-2mm}&=&\hspace{-2mm} -2\rho 1_{\{\alpha < \bar \alpha\}} + ( -\rho\alpha^2 - \rho +\alpha-1)1_{\{\alpha \in [\bar \alpha,\nicefrac{2}{3}]\}}-2\rho 1_{\{\alpha > \nicefrac{2}{3}\}},\\
     \sutil_{w,3}(\alpha;\rho) \hspace{-2mm}&=&\hspace{-2mm} ( -\rho\alpha^2 - \rho +\alpha-1) 1_{\{\alpha < \bar \alpha\}} -2\rho1_{\{\alpha \geq \bar \alpha\}}, \text{ and } \\
    \sutil_{2}(\alpha;\rho) \hspace{-2mm}&=&\hspace{-2mm} ( -1-\rho+2\rho\alpha-2\rho\alpha^2)1_{\{\alpha \leq \nicefrac{1}{2}\}} + ( -1 - \nicefrac{\rho}{2})1_{\{\alpha > \nicefrac{1}{2}\}}.\label{fun_g_w}
\end{eqnarray*}
The explicit expressions of the comparison functions are as follows:
\begin{eqnarray}
    g_b(\alpha;\rho) &=& 
\begin{cases}
    1-\rho-2\rho\alpha+2\rho\alpha^2, & \text{ if } \alpha \in (0,\bar \alpha),
    \label{eqn_g_b_fun}
    \\
         \rho\alpha^2-2\rho\alpha+\alpha,                    & \text{ if } \alpha \in [\bar \alpha,\nicefrac{1}{2}],\\
         -\rho\alpha^2-\nicefrac{\rho}{2}+\alpha, & \text{ if } \alpha \in (\nicefrac{1}{2},\nicefrac{2}{3}],  \\
    1-\nicefrac{3\rho}{2}, & \text{ if } \alpha \in (\nicefrac{2}{3},1].
\end{cases} \label{eqn_gb_fun}\\
g_w(\alpha;\rho)  &=& 
\begin{cases}
    \rho\alpha^2+\alpha-2\rho\alpha, & \text{ if } \alpha \in (0,\bar \alpha), \\
        1-\rho-2\rho\alpha+2\rho\alpha^2,                    & \text{ if } \alpha \in [\bar \alpha,\nicefrac{1}{2}],\\
    1-\nicefrac{(3\rho)}{2}, & \text{ if } \alpha \in (\nicefrac{1}{2},1]. \label{eqn_g_w_fun}
\end{cases}
\end{eqnarray}
\end{document}